\documentclass[reqno]{amsart}

\usepackage[hmargin=35mm,vmargin=35mm]{geometry}
\usepackage{amsmath}
\usepackage{amsfonts}
\usepackage{amsthm}
\usepackage{amssymb}
\usepackage{enumerate}
\usepackage{hyperref}
\usepackage[textsize=scriptsize,colorinlistoftodos]{todonotes}

\newtheorem{thm}{Theorem}[section]
\newtheorem{prop}[thm]{Proposition}
\newtheorem{cor}[thm]{Corollary}
\newtheorem{lem}[thm]{Lemma}
\newtheorem{conj}[thm]{Conjecture}

\theoremstyle{definition}

\newtheorem{rmk}[thm]{Remark}

\newcommand{\CC}{\mathbb{C}}

\newcommand{\ZZ}{\mathbb{Z}}
\newcommand{\RR}{\mathbb{R}}
\newcommand{\QQ}{\mathbb{Q}}

\newcommand{\HH}{\mathcal{H}}

\newcommand{\LL}{\mathcal{L}}

\newcommand{\ul}{\underline}
\newcommand{\D}{\mathcal{D}}

\DeclareMathOperator{\Herm}{Herm}

\DeclareMathOperator{\Sp}{Sp}

\DeclareMathOperator{\CH}{CH}
\DeclareMathOperator{\Sym}{Sym}
\DeclareMathOperator{\tor}{tor}

\DeclareMathOperator{\BB}{BB}
\DeclareMathOperator{\Alb}{Alb}

\DeclareMathOperator{\Span}{Span}

\DeclareMathOperator{\SCH}{SCH}

\begin{document}

\title[Modularity of Special 0-Cycles on Toroidal Compactifications]{Modularity of Special 0-Cycles on Toroidal Compactifications of Shimura Varieties}

\author{Jan Hendrik Bruinier, Eugenia Rosu, and Shaul Zemel}

\date{}

\begin{abstract}
We consider the generating series of appropriately completed 0-dimensional special cycles on a toroidal compactification of an orthogonal or unitary Shimura variety with values in the Chow group. We prove that it is a holomorphic Siegel, respectively Hermitian, modular form.
\end{abstract}

\maketitle

\section{Introduction}

\subsection{Kudla's Modularity Conjecture}

In \cite{HZ}, Hirzebruch and Zagier considered special divisors on Hilbert modular surfaces, and proved the modularity of the generating series with special divisors as coefficients, taking values in the cohomology group. This was further generalized by Gross, Kohnen, and Zagier in \cite{GKZ}, who defined Heegner divisors on compactified modular curves. A more general construction for orthogonal and unitary Shimura varieties was presented by Kudla and Millson in \cite{KM}.

In \cite{Ku}, Kudla proposed a wide reaching program with many important applications. For $r\geq1$ and an appropriate symmetric (resp.\ Hermitian) $r \times r$ matrix $T$, there is a natural special cycle $Z(T)$ of codimension $r$ on the open orthogonal (resp.\ unitary) Shimura variety $X$, with a corresponding Chow class $[Z(T)]\in\CH^{r}(X)$. An essential part of the program consists of showing that generating series having these special cycles as coefficients and taking values in the Chow group $\CH^{r}(X)$ are modular.

Moreover, this is conjectured to be true for suitable extensions of the special cycles $Z(T)$ to $\CH^{r}(X^{\tor})_{\CC}$, where $X^{\tor}$ is a toroidal compactification of $X$. For the orthogonal case this was initially proposed in Problem 3 of \cite{Ku} (see also \cite{Li} for the unitary case). We present both cases together below.
\begin{conj} \label{Kudla}
For any toroidal compactification $X^{\tor}$ of $X$, there exists a natural class $[\widetilde{Z}(T)]\in\CH^{r}(X^{\tor})_{\CC}$ extending the special cycles $Z(T)$ on $X$, such that for $\tau$ in the Siegel upper half-space $\mathcal{H}_{r}$ (resp.\ in the Hermitian upper half-space $\mathbb{H}_{r}$), the generating series
\[
\sum_{T} [\widetilde{Z}(T)] \cdot q^{T}\in\CH^{r}(X^{\tor})_{\CC}[[q]],\qquad\mathrm{with\ }q^{T}:=e^{2\pi i\operatorname{tr}(T\tau)},
\]
is a modular form with respect to the group $\widetilde{\Sp}_{2r}$ (resp.\ $\operatorname{U}_{r,r}$), taking values in $\CH^{r}(X^{\tor})_{\CC}$.
\end{conj}
The statement of Conjecture \ref{Kudla} has the usual meaning, that for every linear functional $l:\CH^{r}(X^{\tor})_{\CC}\to\CC$, the generating series $\sum_{T}l(\widetilde{Z}(T))q^{T}$ is a modular form in the usual sense.

\smallskip

In the case of an open orthogonal or unitary Shimura variety $X$, Kudla and Millson proved in their seminal work \cite{KM} the modularity of the series with values in the cohomology group $H^{2r}(X,\CC)$. For special divisors on orthogonal Shimura varieties, Borcherds extended in \cite{Bo2} the modularity in the Chow group $\CH^{1}(X)_{\CC}$ using the Borcherds lift. Zhang established in \cite{Zh} the modularity in any codimension, conditional on convergence. The first author and Raum completed in \cite{BR} the proof of the modularity conjecture in the Chow group in the orthogonal case. We also mention that \cite{YZZ} proved the modularity of divisors over totally real fields, as well as a similar conditional result for higher codimensional cycles. For open unitary Shimura varieties, the case of divisors was proved by Liu in \cite{Liu} (see also \cite{Ho2} for a proof using Borcherds products). For higher codimensional cycles, the conjecture was proved conditionally on the convergence of the generating series in \cite{Liu}, with cases of the convergence now established by Xia in his extension \cite{Xi} of the results of \cite{BR}.

\smallskip

Much less is known, however, in the case of toroidal compactifications of Shimura varieties. In the orthogonal case, the first and third authors of this paper constructed in \cite{BZ} the appropriate special divisors by investigating the boundary behavior of automorphic Green functions on the open Shimura variety, and proved the modularity of the generating series. Recently, Engel, Greer and Tayou \cite{EGT} showed that the closures $\overline{Z(m)}$ in $X^{\tor}$ of the special divisors $Z(m)$ on $X$ appear as the coefficients of a mixed mock modular form.
A similar result was also proved independently by Garcia in \cite{Ga}.

In the unitary case, for a Shimura variety $X$ corresponding to the unitary group $\operatorname{U}(V)$ for a Hermitian space $V$ of signature $(n,1)$, the modularity of the generating series of divisors follows from the stronger work \cite{BHKRY}, who proved the arithmetic version of Conjecture \ref{Kudla}.

\smallskip

In the current paper we will prove the first higher-codimensional case of Conjecture \ref{Kudla} for a toroidal compactification of a Shimura variety of orthogonal or unitary type. We will construct appropriate special $n$-cycles on the toroidal compactification $X^{\tor}$ and show the modularity of the generating series in $\CH^{n}(X)_{\CC}$, effectively proving Conjecture \ref{Kudla} in case $n=\dim X$.

As in the compact case the results mentioned above provide a complete treatment of Kudla's modularity conjecture (conditional on convergence in the unitary case), throughout this paper we will assume that the open Shimura variety $X$ is non-compact.

\subsection{The Results}

We now present the detailed results of our paper. Let $(V,\langle\cdot,\cdot\rangle)$ to be either a quadratic space over $\QQ$ of signature $(n,2)$ (this will be referred to throughout the paper as \emph{Case 1}), or a Hermitian space of signature $(n,1)$ over an imaginary quadratic field $\mathbb{K}$ (this is \emph{Case 2}), and define, in both cases, the associated reductive group $G$ over $\QQ$ as
\[
G:=\begin{cases} \operatorname{O}(V),\ V/\QQ\text{ orthogonal of signature }(n,2),\text{ i.e., Case 1}, \\ \operatorname{U}(V),\ V/\mathbb{K}\text{ Hermitian of signature }(n,1),\text{ i.e., Case 2}. \end{cases}
\]
Let $L \subseteq V$ be an even lattice of maximal rank, with dual $L^{*}$, and let $\Gamma\subseteq G$ be an arithmetic subgroup associated with $L$ as in Section \ref{ShimVar} below.

We denote by $\mathcal{D}$ the Hermitian symmetric domain associated with the connected component of $G$, namely \[\mathcal{D}:=\begin{cases} \operatorname{SO}_{n,2}^{+}(\mathbb{R})/(\operatorname{SO}_{n}(\mathbb{R})\times\operatorname{SO}_{2}(\mathbb{R})),\text{ in Case 1}, \\ \operatorname{U}_{n,1}(\mathbb{R})/(\operatorname{U}_{n}(\mathbb{R})\times\operatorname{U}_{1}(\mathbb{R})),\quad\text{ in Case 2}. \end{cases}\] Then $(G,\mathcal{D})$ is a Shimura datum and $X(\CC)\simeq\Gamma\backslash\mathcal{D}$ is a connected open Shimura variety.

We fix a toroidal compactification $X^{\tor}$ of $X$, which may be assumed to be smooth (see Section \ref{ShimVar} for more details). For $T\in\Sym_{n}(\QQ)_{>0}$ (resp.\ $T\in\operatorname{Herm}_{n}(\QQ)_{>0}$) positive definite and $\ul{\mu}\in(L^{*}/L)^{n}$, the special cycle $Z(T,\ul{\mu})$ is a set of points in $X$, which we view as a cycle on $X^{\tor}$. We recall from \cite{M} that there is a uniquely defined class, which we denote by $[s]\in\CH^{n}(X^{\tor})_{\CC}$, representing any boundary point on $X^{\tor}$ that maps to a 0-dimensional cusp of the Baily-Borel compactification $X^{\BB}$ of $X$. We then define the associated special cycle class in $\CH^{n}(X^{\tor})_{\CC}$ to be \[Z^{\circ}(T,\ul{\mu}):=Z(T,\ul{\mu})-\deg Z(T,\ul{\mu})[s].\]

The first result of the paper is to establish Conjecture \ref{Kudla} for $r=n$ in the orthogonal case. We denote by $\{\mathfrak{e}_{\ul{\mu}}\}_{\ul{\mu}\in(L^{*}/L)^{n}}$ the standard basis for the group algebra $\CC[(L^{*}/L)^{n}]$, and let $\rho_{L,n}:\widetilde{\Sp}_{2n}(\ZZ)\to\operatorname{GL}\big(\CC[(L^{*}/L)^{n}]\big)$ be the corresponding Weil representation of the metaplectic double cover of the symplectic group $\Sp_{2n}(\ZZ)\subseteq\operatorname{GL}_{2n}(\ZZ)$. Our first main result is as follows.
\begin{thm} \label{modorth}
For $\tau$ in the Siegel half-space $\mathcal{H}_{n}$ of degree $n$, the generating series \[F^{\circ}(\tau):=\sum_{\ul{\mu}\in(L^{*}/L)^{n}}\sum_{T\in\Sym_{n}(\QQ)_{\geq0}}Z^{\circ}(T,\ul{\mu}) \cdot q^{T}\mathfrak{e}_{\ul{\mu}},\qquad q^{T}=e^{2\pi i\operatorname{tr}(T\tau)}\] is a Siegel modular form of degree $n$ with coefficients in $\CH^{n}(X^{\tor})_{\CC}$. It has weight $1+\frac{n}{2}$ and representation $\rho_{L,n}$ with respect to $\widetilde{\Sp}_{2n}(\ZZ)$.
\end{thm}

In the unitary case, for $\tau$ in the Hermitian half-space $\mathbb{H}_{n}$ of degree $n$, we define the generating series
\[
F^{\circ}(\tau):=\sum_{\ul{\mu}\in(L^{*}/L)^{n}}\sum_{T\in\operatorname{Herm}_{n}(\QQ)_{\geq0}}Z^{\circ}(T,\ul{\mu})\cdot q^{T}\mathfrak{e}_{\ul{\mu}},\qquad q^{T}:=e^{2\pi i\operatorname{tr}(T\tau)}.
\]
We similarly denote by $\rho_{L,n}:\operatorname{U}_{n,n}(\ZZ)\to\operatorname{GL}(\CC[(L^{*}/L)^{n}])$ the corresponding Weil representation. Using the result of \cite{Xi}, we obtain the following unconditional result.
\begin{thm} \label{modunit}
Assume that $\mathbb{K}=\QQ(\sqrt{d})$ for $d\in\{-1,-2,-3,-7,-11\}$. Then the generating series $F^{\circ}(\tau)$ is a Hermitian modular form of degree $n$, weight $1+n$ and representation $\rho_{L,n}$ with respect to $\operatorname{U}_{n,n}(\ZZ)$, with values in $\CH^{n}(X^{\tor})_{\CC}$.
\end{thm}

If we assume convergence, then we may remove the assumption on the discriminant in Theorem \ref{modunit}, as follows.
\begin{thm} \label{modafterconv}
If the generating series $F^{\circ}(\tau)$ converges absolutely on $\mathbb{H}_{n}$, then it is a Hermitian modular form of degree $n$, weight $1+n$ and representation $\rho_{L,n}$ with repsect to the unitary group $\operatorname{U}_{n,n}(\ZZ)$, with coefficients in the Chow group $\CH^{n}(X^{\tor})_{\CC}$.
\end{thm}

In the unitary case, we are able to prove the modularity by presenting an isomorphism between the Chow group $\CH^{n}(X^{\tor})^{0}$ of cohomologically trivial cycles on $X^{\tor}$ and the Chow group $\CH^{n}(X)$. The modularity of the generating series $F^{\circ}(\tau)$ follows from the modularity on the open Shimura variety. In the orthogonal case, we establish a similar isomorphism for the cohomologically trivial classes of the special Chow group $\SCH^{n}(X^{\tor})^{0}$ and the special Chow group $\SCH^{n}(X)$ on the open Shimura variety. This requires defining the special Chow group and relating it to the usual Chow group. The modularity follows again from the modularity (with values in the special Chow group) on the open Shimura variety.

\smallskip

We note that the generating series $F^{\circ}(\tau)$ is not expected to vanish in the unitary case (see Appendix~\ref{NonTriv} for more details). In the orthogonal case, the Beilinson--Bloch conjecture would imply that for large $n$ the group of cohomologically trivial Chow classes $\CH^{n}(X)_{\CC}^{0}$ is trivial, thus the entire generating series is expected to vanish identically. However, this is conjectural and we are not aware of any case where this is known. Moreover, one could hope to use the modularity of the generating series in the orthogonal case in order to prove this particular case of the Beilinson--Bloch conjecture.

\smallskip

Our results show that, when taking a suitable projection to the kernel of the degree map $\deg:\CH^{n}(X^{\tor})_{\CC}\to\CC$, the generating series of special cycles of codimension $n$ maps to a holomorphic modular form. A related question is what can be said about the image of the generating series under the degree map, that is,
\[
\sum_{T,\underline{\mu}}\deg(Z(T,\mu)) \cdot q^{T}\mathfrak{e}_{\ul{\mu}}.
\]
If $X$ is already compact, this generating series can be obtained by integrating the corresponding Kudla-Millson theta function, which is a top degree differential form, over $X$. By the Siegel--Weil formula, this integral is given by a holomorphic Siegel Eisenstein series of genus $n$ and weight $1+\frac{n}{2}$ in Case 1, and by an analogous Hermitian Eisenstein series in Case 2. However, if $X$ is non-compact, which is always the case for large $n$, we are outside the range where the Weil convergence criterion for the theta integral applies, and the Siegel-Weil formula will require a suitable regularization. Correspondingly, it is expected that the generating series of the degrees exhibits a mock modular behavior in this case. More precisely, it should be equal to the holomorphic part of a non-holomorphic Siegel (respectively Hermitian) Eisenstein series. In the case where $n=1$, a typical example is given by Zagier's Eisenstein series of weight $\frac{3}{2}$ defined in \cite{Za}, and its generalizations as considered in \cite{Fu}.

\subsection{Structure of the Paper}

In Section \ref{VVJMF} we define the Weil representations, as well as Siegel and Hermitian modular forms. Section \ref{ShimVar} gives the required notions concerning orthogonal Shimura varieties, and in particular we define the toroidal compactifications and the toroidal special 0-cycles. In Section~\ref{Boundary} we discuss the behavior of the special cycles at the boundary. Section \ref{HermMod} establishes the Hermitian modularity of their generating series in the unitary case, while in Section \ref{SiegelMod} we prove the Siegel modularity. Appendix \ref{BoundBeh} contains some technical results about the boundary behavior of cycles on toroidal compactifications of orthogonal Shimura varieties, and Appendix \ref{NonTriv} provides some remarks regarding the non-triviality of the generating series $F^{\circ}(\tau)$.

\subsection*{Acknowledgements}

We thank S. M\"{u}ller-Stach for useful comments involving the discussion showing up in Appendix \ref{NonTriv}. JHB and ER were supported, resp.\ partially supported, by the Deutsche Forschungs-Gemeinschaft (DFG) through the Collaborative Research Centre TRR 326 ``Geometry and Arithmetic of Uniformized Structures'', project number 444845124.

\section{Vector-Valued Jacobi and Modular Forms \label{VVJMF}}

In this section we will define vector-valued Siegel and Hermitian modular forms, and give the definition of the Weil representation in each case.

\subsection{Siegel Modular Forms and the Weil Representation in Case 1}

For $r\geq1$ and a commutative ring $R$, let $\Sym_{r}(R)$ denote the additive group of symmetric $r \times r$ matrices over $R$, and for subrings of $\RR$ the subscripts $>0$ and $\geq0$ denote positive definite and positive semi-definite matrices respectively. Recall that the \emph{Siegel upper half-space} is defined to be \[\mathcal{H}_{r}:=\{\tau\in\Sym_{r}(\CC)\;|\;\mathrm{Im}\tau\in\Sym_{r}(\RR)_{>0}\}.\] Let $\widetilde{\Sp}_{2r}(\ZZ)$ be the metaplectic group, which is a double cover of the symplectic group $\Sp_{2r}(\ZZ)$. Its elements are pairs $[M,\phi]$, where $M=\left(\begin{smallmatrix} A & B \\ C & D \end{smallmatrix}\right)\in\Sp_{2r}(\ZZ)$ and $\phi:\mathcal{H}_{r}\to\CC$ is a holomorphic function such that $\phi(\tau)^{2}=\det(C\tau+D)$. The metaplectic group $\widetilde{\Sp}_{2r}(\ZZ)$ is generated by the elements
\begin{equation} \label{Sp2rgens}
\begin{split}
m(A):= & \left[\left(\begin{smallmatrix} A & 0 \\ 0 & (A^{t})^{-1} \end{smallmatrix}\right),\sqrt{\det A}\right]\mathrm{\ for\ }A\in\operatorname{GL}_{r}(\ZZ), \\ n(B):= & \left[\left(\begin{smallmatrix} I_{r} & B \\ 0 & I_{r} \end{smallmatrix}\right),1\right]\mathrm{\ for\ }B\in\Sym_{r}(\ZZ), \\ \widetilde{w}_{r}:= & \left[\left(\begin{smallmatrix} 0 & -I_{r} \\ I_{r} & 0 \end{smallmatrix}\right),\sqrt{\det\tau}\right],
\end{split}
\end{equation}
where $\sqrt{\det\tau}$ is the square root of $\det\tau$ that takes the matrix $iI$ to $\mathbf{e}\big(\frac{r}{8}\big)$. Here we have introduced the shorthand $\mathbf{e}(z):=e^{2\pi iz}$ for any $z$ in $\CC$ or in $\CC/\ZZ$.

We let $L$ be a lattice inside a quadratic space $V$ of signature $(n,2)$ over $\QQ$ such that the quadratic form is $\ZZ$-valued on $L$, and let $\{\mathfrak{e}_{\ul{\mu}}\}_{\ul{\mu}\in(L^{*}/L)^{r}}$ be the standard basis for the group ring $\CC[(L^{*}/L)^{r}]$. The group $\widetilde{\Sp}_{2r}(\ZZ)$ acts on the group ring $\CC[(L^{*}/L)^{r}]$ via the associated Weil representation, which we denote by $\rho_{L,r}$. The Weil representation $\rho_{L,r}$ is defined by the action of the generators of $\widetilde{\Sp}_{2r}(\ZZ)$ via the formulae
\begin{equation} \label{Weil}
\begin{split}
\rho_{L,r}\big(m(A)\big)\mathfrak{e}_{\ul{\mu}}= & \ \sqrt{\det A}^{2-n}\mathfrak{e}_{\ul{\mu}A^{-1}}, \\ \rho_{L,r}\big(n(B)\big)\mathfrak{e}_{\ul{\mu}}= & \ \mathbf{e}\Big(\operatorname{tr}(Q(\ul{\mu})B)\Big) \mathfrak{e}_{\ul{\mu}}, \\ \qquad\rho_{L,r}(\widetilde{w}_{r})\mathfrak{e}_{\ul{\mu}}= & \ \frac{\mathbf{e}\big(r\frac{2-n}{8}\big)}{|L^{*}/L|^{r/2}}\sum_{\ul{\nu}\in(L^{*}/L)^{r}} \mathbf{e}\big(-\operatorname{tr}\langle\ul{\mu},\ul{\nu}\rangle\big)\mathfrak{e}_{\ul{\nu}},
\end{split}
\end{equation}
where $\langle\ul{\mu},\ul{\nu}\rangle=(\langle\mu_{i},\nu_{j}\rangle)_{1\leq i,j \leq r}$ and $Q(\ul{\mu})=\big(\tfrac{\langle\mu_{i},\mu_{j}\rangle}{2}\big)_{1\leq i,j \leq r} $ are matrices with coefficients in $\QQ/\ZZ$. Note that while division by 2 is not well-defined in $\QQ/\ZZ$, the $i^{\text{th}}$ diagonal entry of $Q(\ul{\mu})$ is $Q(\mu_{i})$, and the off-diagonal ones are multiplied by an even integer in $\operatorname{tr}(Q(\ul{\mu})B)$, making this expression well-defined. This representation comes from a subrepresentation of the Weil representation in the Schr\"{o}dinger model of the ad\'{e}lic metaplectic group $\widetilde{\Sp}_{2r}(\mathbb{\mathbb{A}})$ on the space of Schwartz--Bruhat functions $\mathcal{S}(V_{\mathbb{A}}^{r})$ (see \cite{Zh}). This is the Weil representation in Case 1.

The case $r=1$ in \eqref{Weil} recovers the well-known Weil representation associated with the lattice $L$ for $\widetilde{\operatorname{SL}}_{2}(\ZZ)$, the double metaplectic cover of $\operatorname{SL}_{2}(\ZZ)$. The element $S_{1}:=\widetilde{w}_{1}$ and the translation element $T_{1}:=\left[\left(\begin{smallmatrix} 1 & 1 \\ 0 & 1 \end{smallmatrix}\right),1\right]$ generate $\widetilde{\operatorname{SL}}_{2}(\ZZ)$, and \eqref{Weil} for $\rho_{L}:=\rho_{L,1}$ reduces to
\[\rho_{L}(T_{1})\mathfrak{e}_{\mu}=\mathbf{e}\big(Q(\mu)\big)\mathfrak{e}_{\mu}\quad\mathrm{and}\quad\rho_{L}(S_{1})\mathfrak{e}_{\mu}=\frac{\mathbf{e}\big(\frac{2-n}{8}\big)}{\sqrt{|L^{*}/L|}}\sum_{\nu \in L^{*}/L}\mathbf{e}\big(-\langle\mu,\nu\rangle\big)\mathfrak{e}_{\nu}.\] For more properties of $\rho_{L}$ see, e.g., \cite{Bo1}, \cite{Br}, and \cite{Ze1}.

\medskip

A holomorphic function $F:\mathcal{H}_{g}\to\CC[(L^{*}/L)^{g}]$ is a \emph{Siegel modular form of weight $k\in\frac{1}{2}\ZZ$ and representation $\rho_{L,g}$ with respect to $\widetilde{\Sp}_{2g}(\ZZ)$} if it satisfies \[F\big((A\tau+B)(C\tau+D)^{-1}\big)=\det(C\tau+D)^{k}\rho_{L,g}(M)F(\tau)\] for every $\tau\in\mathcal{H}_{g}$ and $M=\left[\left(\begin{smallmatrix} A & B \\ C & D \end{smallmatrix}\right),\det(C\tau+D)^{1/2}\right]\in\widetilde{\Sp}_{2g}(\ZZ)$. If $g=1$, then the function $F$ is also required to be holomorphic at the cusp at $\infty$.
Such a Siegel modular form $F$ has a Fourier expansion of the form \[F(\tau)=\sum_{\ul{\mu}\in(L^{*}/L)^{g}}\sum_{T\in\Sym_{g}(\QQ)_{\geq0}}c(T,\ul{\mu})\mathbf{e}\big(\operatorname{tr}(T\tau)\big)\mathfrak{e}_{\ul{\mu}},\] with $c(T,\ul{\mu})\in\CC$.

\subsection{Hermitian Modular Forms and the Weil Representation in Case 2}

In the unitary case, let $\mathbb{K}$ be an imaginary quadratic field, with ring of integers $\mathcal{O}_{\mathbb{K}}$ and different $D_{\mathbb{K}}$. For any $\mathcal{O}_{\mathbb{K}}$-algebra $R$, with the extension of the conjugation of $\mathcal{O}_{\mathbb{K}}$ over $\ZZ$ to $R$, we denote by $\operatorname{Herm}_{r}(R)$ the module of Hermitian $r \times r$ matrices over $R$, with a similar meaning for the subscripts $>0$ and $\geq0$ when $R$ is a subring of $\CC$. The integral unitary group is 
\begin{equation} \label{Urrdef}
\operatorname{U}_{r,r}(\ZZ):=\Big\{A\in\operatorname{GL}_{2r}(\mathcal{O}_{\mathbb{K}})\;\Big|\;A\Big(\begin{smallmatrix} 0 & -I_{r} \\ I_{r} & 0 \end{smallmatrix}\Big)A^{*}=\Big(\begin{smallmatrix} 0 & -I_{r} \\ I_{r} & 0 \end{smallmatrix}\Big)\Big\},
\end{equation}
where $A^{*}:=\overline{A}^{t}$ is the transpose-conjugate of $A$. The group $\operatorname{U}_{r,r}(\ZZ)$ is generated, analogously to \eqref{Sp2rgens}, by the elements
\begin{equation} \label{Ugens}
\begin{split}
m(A):= & \left(\begin{smallmatrix} A & 0 \\ 0 & (A^{*})^{-1} \end{smallmatrix}\right)\mathrm{\ for\ }A\in\operatorname{GL}_{r}(\mathcal{O}_{\mathbb{K}}), \\ n(B):= & \left(\begin{smallmatrix} I_{r} & B \\ 0 & I_{r} \end{smallmatrix}\right)\mathrm{\ for\ }B\in\operatorname{Herm}_{r}(\mathcal{O}_{\mathbb{K}}), \\ w_{r}:= & \left(\begin{smallmatrix} 0 & -I_{r} \\ I_{r} & 0 \end{smallmatrix}\right).
\end{split}
\end{equation}

Consider now a Hermitian vector space $V$ of signature $(n,1)$ over $\mathbb{K}/\QQ$ and an $\mathcal{O}_{\mathbb{K}}$-lattice $L \subseteq V$, where the $\QQ$-valued Hermitian form takes integral values on $L$. Then $L$ is a $\ZZ$-lattice, with dual $L^{*}$, and take $\CC[(L^{*}/L)^{r}]$ and $\{\mathfrak{e}_{\ul{\mu}}\}_{\ul{\mu}\in(L^{*}/L)^{r}}$ as before (for more details on this construction see, e.g., Section 1 of \cite{Ze3}). The Weil representation $\rho_{L,r}:\operatorname{U}_{r,r}(\ZZ)\to\operatorname{GL}\big(\CC[(L^{*}/L)^{r}]\big)$ is again given in terms of the generators from \eqref{Ugens} by
\begin{equation} \label{HermWeil}
\begin{split}
\rho_{L,r}\big(m(A)\big)\mathfrak{e}_{\ul{\mu}}= & \ (\det A)^{-n-1}\mathfrak{e}_{\ul{\mu}A^{-1}}, \\ \rho_{L,r}\big(n(B)\big)\mathfrak{e}_{\ul{\mu}}= & \ \mathbf{e}\big(\operatorname{tr}(Q(\ul{\mu})B)\big)\mathfrak{e}_{\ul{\mu}}, \\ \rho_{L,r}(w_{r})\mathfrak{e}_{\ul{\mu}}= & \ \frac{\gamma_{V,f}}{|L^{*}/L|^{r/2}}\sum_{\ul{\nu}\in(L^{*}/L)^{r}}\mathbf{e}\big(-\frac{1}{2}\operatorname{tr}_{\mathbb{K}/\QQ}\langle\ul{\mu},\ul{\nu}\rangle\big) \mathfrak{e}_{\ul{\nu}},
\end{split}
\end{equation}
where $\gamma_{V,f}$ is an 8th root of unity called the Weil index (restricted to the finite places). Note that $\langle\ul{\mu},\ul{\nu}\rangle$ and $Q(\ul{\mu})$ in \eqref{HermWeil} are now $\mathbb{K}/D_{\mathbb{K}}^{-1}$-valued, but since $Q(\ul{\mu})$ and $B$ are Hermitian, the arguments of the exponents are well-defined in $\QQ/\ZZ$. This representation is constructed from the classical ad\'elic Weil representation in the Schr\"{o}dinger model, analoguously to the orthogonal case treated by \cite{Zh}. We note in particular that the ad\'elic Weil representation comes with an associated quadratic character $\chi:\mathbb{A}_{\mathbb{K}}^{\times}/\mathbb{K}^{\times}\to\CC^{\times}$. However, here we are interested only in the infinite place $\chi_{\infty}:z\mapsto(z/\overline{z})^{n+1}$, which is independent of $\mathbb{K}$, thus the choice of character does not play a role in the representation $\rho_{L,r}$.

We will again write $\rho_{L}:=\rho_{L,1}$ for the Weil representation of $\operatorname{U}_{1,1}(\ZZ)$. We recall from \cite{Ho1} and \cite{Ze3} that the group $\operatorname{U}_{1,1}(\ZZ)$ equals $\mu_{\mathbb{K}}\cdot\operatorname{SL}_{2}(\ZZ)$, where $\mu_{\mathbb{K}}$ is the group of roots of unity of $\mathbb{K}$ viewed as scalar matrices. If the discriminant of $\mathbb{K}$ is smaller than $-4$, then $\operatorname{U}_{1,1}(\ZZ)=\operatorname{SL}_{2}(\ZZ)$. Otherwise an extension of the representation $\rho_{L}$ from $\operatorname{SL}_{2}(\ZZ)$ to $\operatorname{U}_{1,1}(\ZZ)$ is described explicitly in Proposition 2.3 of \cite{Ze3}.

Given $g\geq1$, we note that for $\tau\in\operatorname{M}_{g}(\CC)$ the matrix $\frac{\tau-\tau^{*}}{2i}$ is Hermitian, and we consider the \emph{Hermitian upper half-space} \[\mathbb{H}_{g}:=\left\{\tau\in\operatorname{M}_{g}(\CC)\;|\;\frac{\tau-\tau^{*}}{2i}\in\operatorname{Herm}_{g}(\mathbb{K})_{>0}\right\}.\]
Note that, for $g=1$, we have $\mathbb{H}_{1}=\mathcal{H}_{1}=\mathcal{H}$, while in general the Hermitian half-space $\mathbb{H}_{g}$ contains the Siegel upper half-space $\mathcal{H}_{g}$.

Recall that the integral unitary group $\operatorname{U}_{g,g}(\ZZ)$ is defined in \eqref{Urrdef} as the stabilizer of the matrix $\left(\begin{smallmatrix} 0 & I_{g} \\ -I_{g} & 0 \end{smallmatrix}\right)$ in $\operatorname{GL}_{2g}(\mathcal{O}_{\mathbb{K}})$, and let $\operatorname{U}_{g,g}(\RR)$ be its stabilizer inside $\operatorname{GL}_{2g}(\CC)$. Then $\mathbb{H}_{g}$ is the Hermitian symmetric domain associated with the this Lie group, and the integral group $\operatorname{U}_{g,g}(\ZZ)$ admits the Weil representation $\rho_{L,g}:\operatorname{U}_{g,g}(\ZZ)\to\operatorname{GL}\big(\CC[(L^{*}/L)^{g}]\big)$ defined in \eqref{HermWeil}.

We call a holomorphic function $F:\mathbb{H}_{g}\to\CC[(L^{*}/L)^{g}]$ a \emph{Hermitian modular form of weight $k$ and representation $\rho_{L,g}$ with respect to $\operatorname{U}_{g,g}(\ZZ)$} if it satisfies \[F\big((A\tau+B)(C\tau+D)^{-1}\big)=\det(C\tau+D)^{k}\rho_{L,g}\big(\begin{smallmatrix} A & B \\ C & D \end{smallmatrix}\big)F(\tau)\] for all $\tau\in\mathbb{H}_{g}$ and $\left(\begin{smallmatrix} A & B \\ C & D \end{smallmatrix}\right)\in\operatorname{U}_{g,g}(\ZZ)$. If $g=1$, then the function $F$ is also required to be holomorphic at the cusps. A Hermitian modular form $F$ also admits a Fourier expansion \[F(\tau)=\sum_{\ul{\mu}\in(L^{*}/L)^{g}}\sum_{T\in\operatorname{Herm}_{g}(\mathbb{K})_{\geq0}}c(T,\ul{\mu})\mathbf{e}\big(\operatorname{tr}(T\tau)\big)\mathfrak{e}_{\ul{\mu}},\] where $c(T,\ul{\mu})\in\CC$.

\section{Shimura Varieties and their Special Cycles \label{ShimVar}}

In this section we will define the orthogonal and unitary Shimura varieties and their compactifications, as well as the special cycles on the toroidal compactifications.

\subsection{Basic Definitions}

We will consider here Cases 1 and 2 simultaneously. We define $\mathbb{K}$ to be $\QQ$ in Case 1, and the underlying imaginary quadratic field in Case 2, with a fixed embedding into $\CC$. We recall the ring of integers $\mathcal{O}_{\mathbb{K}}$ and the inverse different $D_{\mathbb{K}}^{-1}$, both of which equal $\ZZ$ in Case 1.

We take $\big(V,\langle\cdot,\cdot\rangle\big)$ to be a vector space over $\mathbb{K}$ which is quadratic of signature $(n,2)$ in Case 1 and Hermitian of signature $(n,1)$ in Case 2 (in particular, the pairing is non-degenerate in both cases). In our convention in Case 2, the pairing is $\mathbb{K}$-linear in the first variable and conjugate-linear in the second variable. We will write $V_{\RR}$ for $V\otimes_{\QQ}\RR$ in both cases, where it carries an $\RR$-valued quadratic form in Case 1 and a $\CC$-valued Hermitian pairing in Case 2.

We denote by $G(V_{\RR})$ the group of isometries of $V_{\RR}$, namely $\operatorname{O}(V_{\RR})$ in Case 1 and $\operatorname{U}(V_{\RR})$ in Case 2, and write $G(V_{\RR})^{\circ}$ for the connected component of the identity in $G(V_{\RR})$. This means that $G(V_{\RR})^{\circ}$ is isomorphic to $\operatorname{SO}_{n,2}^{+}(\RR)$ in Case 1 and to $\operatorname{U}_{n,1}(\RR)$ in Case 2, and thus its symmetric domain, the \emph{Grassmannian} $\mathcal{D}$, is isomorphic to
\[G(V_{\RR})^{\circ}/K\simeq\begin{cases} \operatorname{SO}_{n,2}^{+}(\RR)/[\operatorname{SO}_{n}(\RR)\times\operatorname{SO}_{2}(\RR)], & \text{in Case 1} \\ \operatorname{U}_{n,1}(\RR)/[\operatorname{U}_{n}(\RR)\times\operatorname{U}_{1}(\RR)], & \text{in Case 2}, \end{cases}\] with $K \subseteq G(V_{\RR})^{\circ}$ a maximal compact subgroup. For an explicit description of $\mathcal{D}$, in Case 1 the Grassmannian $\mathcal{D}$ is a connected component of
\begin{equation} \label{bunGrO}
\big\{z\in\mathbb{P}(V_{\CC})\big|\;\langle z,z \rangle=0,\;\langle z,\overline{z} \rangle<0\big\}\subseteq\mathbb{P}(V_{\CC})
\end{equation}
which is determined by fixing the orientations of the 2-dimensional subspaces of $V_{\RR}$. In Case 2 we get
\begin{equation} \label{bunGrU}
\mathcal{D}=\big\{z\in\mathbb{P}(V_{\RR})\big|\;\langle z,z \rangle<0\big\}\subseteq\mathbb{P}(V_{\RR}),
\end{equation}
which is isomorphic to a complex ball of dimension $n$.

\medskip

For $\ul{x}=(x_{1},\dots,x_{r}) \in V^{r}$, we define $Q(\ul{x})=(\frac{1}{2}\langle x_{j},x_{i} \rangle)_{1 \leq i,j \leq r}$ to be the intersection matrix. For $r=1$, we get the quadratic form $Q(y):=\frac{1}{2}\langle y,y \rangle\in\QQ$ for any $y \in V$.

A finitely generated $\mathcal{O}_{\mathbb{K}}$-module $L$ of full rank inside $V$ is called an \emph{even lattice} if $Q(\lambda) \in D_{\mathbb{K}}^{-1}$ for every $\lambda \in L$ (as $Q(y)\in\QQ$ for $y \in V$, this means that if $\lambda \in L$ then $Q(\lambda)$ must lie in $\ZZ$ in Case 1 and in $\frac{1}{2}\ZZ$ in Case 2---see Remark \ref{Hermtoquad} below). For such $L$, its \emph{dual lattice} is $L^{*}:=\{\lambda \in V\;|\;\langle\lambda, L\rangle \subseteq D_{\mathbb{K}}^{-1}\} \subseteq V$. We also take a finite index subgroup $\Gamma$ of the \emph{stable orthogonal group} $\operatorname{O}(L)$ of $L$, namely
\[
\Gamma\subseteq\big\{\gamma \in G(V_{\RR})^{\circ}\;\big|\;\gamma y-y \in L\ \forall y \in L^{*}\big\}=:\operatorname{O}(L).
\] 
The action of the subgroup $\Gamma$ on $\mathcal{D}$ is properly discontinuous, and the quotient yields the open connected Shimura variety \[X:=\Gamma\backslash\mathcal{D},\] which is a complex orbifold of dimension $n$.

\begin{rmk} \label{Hermtoquad}
In Case 2, the pairing defined by $(x,y):=\operatorname{tr}_{\mathbb{K}/\QQ}\langle x,y \rangle$ for $x, y\in V$ endows $V$, as a vector space over $\QQ\subseteq\mathbb{K}$, with the structure of a rational quadratic space. Our condition on $L$ to be an even lattice is equivalent to it becoming an even lattice in this restriction of scalars, and the resulting dual lattice $\{\lambda \in V\;|\;\operatorname{tr}_{\mathbb{K}/\QQ}\langle\lambda, L\rangle\subseteq\ZZ\} \subseteq V$ coincides with our $L^{*}$.
\end{rmk}

\subsection{Boundary and Compactifications}

In Case 1, the boundary of \eqref{bunGrO} is obtained by replacing the inequality there by an equality, namely we have \[\partial\mathcal{D}\subseteq\big\{z\in\mathbb{P}(V_{\CC})\big|\;\langle z,z \rangle=0,\;\langle z,\overline{z} \rangle=0\big\}.\] An element $z\in\partial\mathcal{D}$ lifts to an element $\tilde{z} \in V_{\CC}$ satisfying these equalities. The real and imaginary parts of $\tilde{z}$ span an isotropic space $I_{z} \subseteq V_{\RR}$ of dimension 1 or 2, which is independent on the choice of $\tilde{z}$. If $I_{z}$ is an isotropic plane, then $z$ will be in the boundary $\partial\mathcal{D}$ when the resulting orientation on $I_{z}$ matches the one that yields the connected component $\mathcal{D}$. For any isotropic plane $J \subseteq V_{\RR}$, the set of $z\in\partial\mathcal{D}$ with $I_{z}=J$ is a 1-dimensional complex manifold which is isomorphic to $\mathcal{H}$. In the case of $I_{z}$ an isotropic line, $z$ lies in $\partial\mathcal{D}$, and for an isotropic line $J \subseteq V_{\RR}$ there is a unique $z\in\partial\mathcal{D}$ with $I_{z}=J$, yielding a 0-dimensional boundary component of $\mathcal{D}$.

In Case 2, the boundary $\partial\mathcal{D}$ of $\mathcal{D}$ from \eqref{bunGrU} is similarly given by the corresponding equality, namely \[\partial\mathcal{D}=\big\{z\in\mathbb{P}(V_{\RR})\big|\;\langle z,z \rangle=0\big\}.\] For an element $z\in\partial\mathcal{D}$, choose any lift $\tilde{z} \in V_{\RR}$ of $z$, yielding the associated isotropic line $I_{z} \subseteq V_{\RR}$, which is independent of $\tilde{z}$. For every such isotropic line $J \subseteq V_{\RR}$ there is an associated 0-dimensional boundary component of $\mathcal{D}$, consisting of the unique $z\in\partial\mathcal{D}$ with $I_{z}=J$.

We call an element $z\in\partial\mathcal{D}$ \emph{rational} if the corresponding isotropic subspace $I_{z} \subseteq V_{\RR}$ is rational, meaning that $I_{z}=I_{\RR}$ for an isotropic subspace $I \subseteq V$ over $\mathbb{K}$ (with $I_{\RR}$ again meaning $I\otimes_{\QQ}\RR$). The \emph{Baily--Borel compactification} $\mathcal{D}^{\BB}$ of $\mathcal{D}$ is defined to be $\mathcal{D}\cup\{z\in\partial\mathcal{D}\;|\;z\text{ is rational}\}$. Moveover, as the action of $\Gamma$ naturally extends to $\mathcal{D}^{\BB}$, we define the \emph{Baily--Borel compactification} of $X$ to be \[X^{\BB}:=\Gamma\backslash\mathcal{D}^{\BB},\] which is the complete algebraic variety over $\CC$ yielding the minimal algebraic compactification of $X$. The boundary components of both $\mathcal{D}^{\BB}$ and $X^{\BB}$ are called \emph{cusps}, and they can be 0-dimensional or 1-dimensional, with only the first type existing in Case 2. From Meyer's Theorem we deduce, in Case 1, that $X^{\BB}$ has 0-dimensional cusps when $n\geq3$ and also 1-dimensional cusps once $n\geq5$, while $X^{\BB}$ contains 0-dimensional cusps in Case 2 wherever $n\geq2$. Note that if there are 1-dimensional cusps on $X^{\BB}$ in Case 1, then their closures contain 0-dimensional cusps.

\medskip

In general, the cusps of the Baily--Borel compactification $X^{\BB}$ are highly singular. One way to resolve these singularities is by replacing it by a toroidal compactification $X^{\tor}$, also containing $X$ as an open subvariety. In Case 2 the toroidal compactification $X^{\tor}$ is canonical, and smooth if $\Gamma$ is neat (see, e.g., \cite{Liu}). In Case 1 the details of the construction, which depend on certain combinatorial data $\Sigma$ (a \emph{cone decomposition}), are given in \cite{Fi} and \cite{BZ}, and if $\Gamma$ is neat then it is possible to choose $\Sigma$ such that $X^{\tor}$ is smooth.

Moreover, the toroidal compactification $X^{\tor}$ comes with a natural map to the minimal compactification $X^{\BB}$ that we denote by
\begin{equation} \label{Xtornot}
\pi:X^{\tor} \to X^{\BB}.
\end{equation}
In Case 1, the preimage of a 0-dimensional cusp depends on the cone decomposition $\Sigma$, while that of a 1-dimensional cusp is canonical (i.e., independent of $\Sigma$) and resembles an open Kuga--Sato variety (see, e.g., \cite{Ze2}). In Case 2, the preimage of any cusp is a (possibly reducible) Abelian variety of dimension $n-1$, which is isogenous to the product of $n-1$ elliptic curves with CM from $\mathbb{K}$.

\subsection{Shimura Subvarieties}

Fix $0 \leq r \leq n$, and take an $r$-tuple $\ul{x}:=(x_{1},\ldots,x_{r}) \in V^{r}$ such that the subspace $\mathbb{K}\ul{x} \subseteq V$ spanned by the $x_{i}$'s over $\mathbb{K}$ is positive definite. We denote by $r(\ul{x})$ the dimension of $\mathbb{K}\ul{x}$ over $\mathbb{K}$, and define the complementary vector space \[V_{\ul{x}}:=\langle\ul{x}\rangle^{\perp}=\{v \in V\;|\;\langle v,x_{i} \rangle=0\ \forall\ 1 \leq i \leq r\},\] which is quadratic of signature $\big(n-r(\ul{x}),2\big)$ in Case 1 and Hermitian of signature $\big(n-r(\ul{x}),1\big)$ in Case 2, with a similar extension of scalars $V_{\ul{x},\RR} \subseteq V_{\RR}$. The group of isometries $G(V_{\ul{x},\RR})$ can be identified with the pointwise stabilizer of $\ul{x}$ in $G(V_{\RR})$. Moreover, $G(V_{\ul{x},\RR})$ is isomorphic to $\operatorname{O}_{n-r(x),2}(\RR)$ in Case 1 and to $\operatorname{U}_{n-r(x),1}(\RR)$ in Case 2. Its symmetric space, or more precisely that of its identity component $G(V_{\ul{x},\RR})^{\circ}$, is a complex submanifold of dimension $n-r(\ul{x})$, which can be identified with the submanifold \[\mathcal{D}_{\ul{x}}:=\{z\in\mathcal{D}\;|\;\langle z,x_{i} \rangle=0\ \forall\ 1 \leq i \leq r\}\subseteq\mathcal{D}.\]

The natural embedding of $V_{\ul{x}}$ into $V$ also takes the isotropic subspaces of $V_{\ul{x}}$ to isotropic subspaces of $V$. Thus the Baily--Borel compactification $\mathcal{D}_{\ul{x}}^{\BB}$ embeds into $\D^{\BB}$, and in fact we have
\begin{equation} \label{perpDBB}
\mathcal{D}_{\ul{x}}^{\BB}:=\{z\in\D^{\BB}\;|\;\langle z,x_{i} \rangle=0\ \forall\ 1 \leq i \leq r\},
\end{equation}
extending the relation for the open domains.

The group acting on $\mathcal{D}_{\ul{x}}$ and $\mathcal{D}_{\ul{x}}^{\BB}$ is the pointwise stabilizer of $\ul{x}$ in $\Gamma$, namely \[\Gamma_{\ul{x}}:=\Gamma \cap G(V_{\ul{x},\RR})^{\circ}=\{\gamma\in\Gamma\;|\;\gamma x_{i}=x_{i}\ \forall\ 1 \leq i \leq r\}.\] The corresponding open Shimura subvariety and its Baily--Borel compactification are
\[
X_{\ul{x}}:=\Gamma_{\ul{x}}\backslash\mathcal{D}_{\ul{x}} \qquad\mathrm{and}\qquad X_{\ul{x}}^{\BB}:=\Gamma_{\ul{x}}\backslash\mathcal{D}_{\ul{x}}^{\BB},
\]
respectively. The closed embeddings $\mathcal{D}_{\ul{x}}\hookrightarrow\mathcal{D}$ and of $\mathcal{D}_{\ul{x}}^{\BB}\hookrightarrow\D^{\BB}$ yield natural proper maps
\begin{equation} \label{embvecx}
\iota_{\ul{x}}:X_{\ul{x}} \to X\qquad\mathrm{and}\qquad\iota_{\ul{x}}^{\BB}:X_{\ul{x}}^{\BB} \to X^{\BB},
\end{equation}
which are injective when $\Gamma$ is neat.

For the toroidal compactifications, we recall that in Case 2 the construction is canonical, so that we also obtain the variety $X_{\ul{x}}^{\tor}$. In Case 1, \cite{H} shows that for any cone decomposition $\Sigma$ for $X$, there is a natural choice of a cone decomposition $\Sigma_{\ul{x}}$ for $X_{\ul{x}}$. In both cases we have, as in \eqref{Xtornot}, a natural projection map $\pi_{\ul{x}}:X_{\ul{x}}^{\tor} \to X_{\ul{x}}^{\BB}$. Moreover, the map $\iota_{\ul{x}}$ from \eqref{embvecx} extends to a closed, proper map between the resulting compactifications, which we denote by
\begin{equation} \label{iotator}
\iota_{\ul{x}}^{\tor}:X_{\ul{x}}^{\tor} \to X^{\tor}.
\end{equation}
This map is also injective when $\Gamma$ is neat, and it commutes with the maps $\iota_{\ul{x}}^{\BB}$, $\pi$, and its counterpart $\pi_{\ul{x}}$. As $X_{\ul{x}}^{\tor}$ has codimension $r(\ul{x})$ in $X^{\tor}$, we obtain an induced pushforward map
\begin{equation} \label{pushforward}
\iota_{\ul{x},*}^{\tor}:\CH^{1}(X_{\ul{x}}^{\tor})\to\CH^{r(\ul{x})+1}(X^{\tor}).
\end{equation}

\medskip

The \emph{special cycle} $Z(\ul{x})$ on the open Shimura variety $X$ is the image of the canonical morphism $\iota_{\ul{x}}:X_{\ul{x}} \to X$. We will denote by $Z(\ul{x})$ also the closure of the special cycle in $X^{\BB}$ and by $\overline{Z(\ul{x})}$ its closure in $X^{\tor}$. The latter coincides (through the construction from \cite{H} in Case 1 and by canonicity in Case 2) with the image of $\iota_{\ul{x}}^{\tor}$ from \eqref{iotator}. We shall also write $[Z(\ul{x})]$ and $[\overline{Z(\ul{x})}]$ for the corresponding classes in the Chow group $\CH^{r(\ul{x})}(X)$ or $\CH^{r(\ul{x})}(X^{\tor})$.

\subsection{Special 0-cycles}

On the open Shimura variety $X$, special 0-cycles are obtained by intersecting the cycle $Z(\ul{x})$ of codimension $r(\ul{x})$, with $n-r(\ul{x})$ sections of the line bundle $\LL^{\vee}$, the dual of the \emph{tautological bundle} $\LL$ on $X$.

We are interested in defining special $0$-cycles on the toroidal compactification $X^{\tor}$ in a way that is compatible with the pushforward coming from the embedding of $X$ into $X^{\tor}$. We denote by $\LL_{\tor}$ the tautological bundle on $X^{\tor}$, and its dual by $\LL_{\tor}^{\vee}$. Proposition \ref{dimbd} from Appendix \ref{BoundBeh} shows that there is a choice of successive sections of $\LL_{\tor}^{\vee}$ such that the intersection of $\overline{Z(\ul{x})}$ with the divisors of these sections is a collection of points that lie in $X$. For every $\ul{x}\in V^n$ we make such a choice, and write, by a slight abuse of notation,
\[
\overline{Z(\ul{x})}\cdot(\LL_{\tor}^{\vee})^{n-r(\ul{x})}
\] to be the resulting 0-cycle on $X^{\tor}$. By restricting everything to $X$, so that the sections are now of $\LL^{\vee}$, we obtain a 0-cycle on $X$ that we denote similarly by $Z(\ul{x})\cdot(\LL^{\vee})^{n-r(\ul{x})}$. It is clear that pushing forward $Z(\ul{x})\cdot(\LL^{\vee})^{n-r(\ul{x})}$ from $X$ to $X^{\tor}$ then yields $\overline{Z(\ul{x})}\cdot(\LL_{\tor}^{\vee})^{n-r(\ul{x})}$.

We note here that different choices of secions will give different cycles. However, we are only interested in the images of these cycles in the Chow groups $\CH^{n}(X)$ and $\CH^{n}(X^{\tor})$ or in the special Chow groups defined in Section \ref{SiegelMod}. It follows from Proposition \ref{dimbd} that every such cycle represents the class $[Z(\ul{x})]\cdot[\LL^{\vee}]^{n-r(\ul{x})}$ or $[\overline{Z(\ul{x})}]\cdot[\LL_{\tor}^{\vee}]^{n-r(\ul{x})}$, independently of the choices.

Consider now a matrix $T\in\Sym_{n}(\QQ)_{\geq0}$ in Case 1 (resp.\ $T\in\operatorname{Herm}_{n}(\QQ)_{\geq0}$ in Case 2) and an element $\ul{\mu}\in(L^{*}/L)^{n}$, and define the special cycle $Z(T,\ul{\mu})$ on $X$ to be
\begin{equation} \label{spcycop}
Z(T,\ul{\mu})=\sum_{\ul{x}\in\Gamma \backslash L_{T,\ul{\mu}}}Z(\ul{x})\cdot(\LL^{\vee})^{n-r(\ul{x})},
\end{equation}
where $L_{T,\ul{\mu}}=\{\ul{x}\in\ul{\mu}+L^{n}\;|\;Q(\ul{x})=T\}$.

For the rest of this paper we will assume that $X$ is non-compact. Let $s$ be a point on the boundary of $X^{\tor}$ which maps to a cusp of dimension 0 on $X^{\BB}$. Then for $T\in\Sym_{n}(\QQ)_{\geq0}$ (resp.\ $T\in\Herm_{n}(\mathbb{K})_{\geq0}$) and $\ul{\mu}\in(L^{*}/L)^{n}$ we can define the \emph{toroidal special 0-cycle}
\begin{equation} \label{spcyctor}
Z^{\circ}(T,\ul{\mu})
=
\sum_{\ul{x}\in\Gamma \backslash L_{T,\ul{\mu}}}\overline{Z(\ul{x})}\cdot(\LL_{\tor}^{\vee})^{n-r(\ul{x})}-\deg\big(\overline{Z(\ul{x})}\cdot(\LL_{\tor}^{\vee})^{n-r(\ul{x})}\big)[s]\in\CH^{n}(X^{\tor})_{\QQ}.
\end{equation}
From a result of \cite{M} (see Theorem \ref{pt0cusp} in Appendix \ref{BoundBeh}), when $s$ maps to a cusp of dimension 0 in $X^{\BB}$, the class $[s]\in\CH^{n}(X^{\tor})_{\QQ}$ does not depend on the choice of the point $s$.

We note that when $T$ is positive definite, $Z(T,\ul{\mu})$ is a naturally defined special 0-cycle on $X$. If we identify it with its pushforward to $X^{\tor}$, then
\[
Z^{\circ}(T,\ul{\mu})=Z(T,\ul{\mu})-\deg(Z(T,\ul{\mu}))[s].
\]
An alternative way to define the special cycles from \eqref{spcyctor} for all positive semi-definite $T$ is by using the toroidal divisors from \cite{BZ} in Case 1 and from \cite{BHKRY} in Case 2. Corollary \ref{altbys} yields the equivalence of these definitions.

Finally, using the special cycles from \eqref{spcyctor}, we define the formal generating series
\begin{equation} \label{unifgenser}
F^{\circ}(\tau)=\sum_{\ul{\mu}\in (L^{*}/L)^{n}}\sum_{T}Z^{\circ}(T,\ul{\mu})q^{T}\mathfrak{e}_{\mu},\quad q^{T}:=\mathbf{e}\big(\operatorname{tr}(T\tau)\big),
\end{equation}
where $\tau$ is the Siegel upper half space $\mathcal{H}_{n}$ and $T\in\Sym_{n}(\QQ)_{\geq0}$ in Case 1, while in Case 2 the variable $\tau$ is from the Hermitian upper half space $\mathbb{H}_{n}$, and $T$ is taken from $\operatorname{Herm}_{n}(\mathbb{K})_{\geq0}$.

\section{Intersection with the Boundary \label{Boundary}}

We recall that the toroidal compactification $X_{\ul{x}}^{\tor}$ of the variety $X_{\ul{x}}$ consists of the union of the open variety $X_{\ul{x}}$ and a finite collection of boundary divisors. In Case 1, we write
\[
X_{\ul{x}}^{\tor}=X_{\ul{x}}\cup\bigcup B_{\ul{x},I,\omega}\cup\bigcup B_{\ul{x},J},
\] 
where the boundary divisor $B_{\ul{x},J}$ corresponds to the primitive isotropic sublattice $J \subseteq L_{\ul{x}}$ of rank 2, while $B_{\ul{x},I,\omega}$ is the boundary divisor associated with the primitive isotropic sublattice $I \subseteq L_{\ul{x}}$ of rank 1 together with a generator $\omega$ of an internal ray in the rational cone decomposition of the cone associated with $I$. The unions are over representing sets of such $J$'s, $I$'s, and $\omega$'s modulo the action of the group $\Gamma_{\ul{x}}$ (see \cite{Fi} and \cite{BZ} for further details). The boundary divisors of the toroidal compactification $X^{\tor}$ arise from the case $r=0$ and are denoted similarly, with the subscript $\ul{x}$ omitted.

In Case 2, the canonical toroidal compactification $X_{\ul{x}}^{\tor}$ can be written as \[X_{\ul{x}}^{\tor}=X_{\ul{x}}\cup\bigcup B_{I,\ul{x}},\] where $I$ runs over a set of representatives for the action of $\Gamma_{\ul{x}}$ on the set of primitive isotropic submodules of $L_{\ul{x}}$, and $B_{\ul{x},I}$ is the canonical boundary divisor associated with $I$. Note that $B_{\ul{x},I}$ is an Abelian variety of dimension $n-r(\ul{x})-1$, lying over the corresponding cusp of $X_{\ul{x}}^{\BB}$. When $r=0$ and $\ul{x}$ is trivial, we obtain the divisors $B_{I}$ on the full toroidal variety $X^{\tor}$.

To ease certain arguments below, we replace each isotropic module $I$ or $J$ in these definitions by the isotropic subspace of $V$ spanned by it, and use the same notation $I$ or $J$ for this vector space as well.

\medskip

Note that in principle, a cycle of $X^{\BB}$ may intersect a 1-dimensional cusp and not contain it fully. However, this will never be the case for the image of the map $\iota_{\ul{x}}^{\BB}$ from \eqref{embvecx}, as we will show below. This is due to the observation that a 1-dimensional cusp $s$ of $X^{\BB}$ is an open modular curve, thus its own cusps (which correspond to 0-dimensional cusps on the boundary of $X$) are not contained in $s$, but only in its closure.

For a cusp $s$ of $X^{\BB}$, we denote by $\widetilde{s}$ a cusp of $\mathcal{D}^{\BB}$ that projects onto $s$, and denote by $J_{s} \subseteq V$ the isotropic subspace of $V$ corresponding to $\widetilde{s}$.
\begin{lem} \label{cuspsvecx}
The cusp $s$ of $X^{\BB}$ intersects the image of $X_{\ul{x}}^{\BB}$ under $\iota_{\ul{x}}^{\BB}$ if and only if there is an element $\gamma\in\Gamma$ such that $\gamma J_{s} \subseteq V_{\ul{x}}$. When this happens, the image of $X_{\ul{x}}^{\BB}$ contains all of $s$.
\end{lem}

\begin{proof}
From the definition in \eqref{perpDBB}, $\mathcal{D}_{\ul{x}}^{\BB}$ contains only cusps that correspond to isotropic subspaces that are perpendicular to $\ul{x}$, and such cusps are fully contained in $\mathcal{D}_{\ul{x}}^{\BB}$. Moreover, given $z\in\mathcal{D}^{\BB}$, its image in $X^{\BB}$ lies in $\iota^{\BB}_{\ul{x}}(X_{\ul{x}})$ if and only if there exists $\gamma\in\Gamma$ such that $\gamma z\in\mathcal{D}_{\ul{x}}^{\BB}$, i.e., such that $\gamma z\in\mathbb{P}(V_{\CC})$ is perpendicular to $\ul{x}$.

The result is clear when $s$ is a 0-dimensional cusp, since any lift of the corresponding point $z\in\mathcal{D}^{\BB}$ spans the associated space $J_{s}$ over $\mathbb{K}$. Now consider $s$ to be a 1-dimensional cusp of $X^{\BB}$ in Case 1, so that $J_{s}$ is 2-dimensional. For every point $z\in\mathcal{D}^{\BB}$ that lies in the cusp $\widetilde{s}$ associated with $J_{s}$, and every lift $\tilde{z}$ of $z$ to $V_{\CC}$, the real and imaginary parts of $\tilde{z}$ form a basis for $J_{s}$, and thus $\langle z,\ul{x} \rangle=0$ is equivalent to $\ul{x}$ being perpendicular to all of $J_{s}$. In particular the cusp $\widetilde{s}$ of $\mathcal{D}^{\BB}$ is fully contained in $\mathcal{D}_{\ul{x}}^{\BB}$ when $J \subseteq V_{\ul{x}}$, and it is disjoint from $\mathcal{D}_{\ul{x}}^{\BB}$ otherwise. Adding the action of $\Gamma$ as in the previous paragraph yields the desired result. This proves the lemma.
\end{proof}
We comment on the case where the 2-dimensional space $J_{s}$ intersects $V_{\ul{x}}$ in a 1-dimensional subspace. Then $I:=J_{s} \cap V_{\ul{x}} \subseteq J_{s}$ is a rational isotropic line, and the corresponding 0-dimensional cusp $\widetilde{s}_{0}$ of $\mathcal{D}^{\BB}$ lies in the closure of the 1-dimensional cusp $\tilde{s}$ associated with $J_{s}$, but not inside the cusp $\tilde{s}$ itself. It follows that $\mathcal{D}_{\ul{x}}^{\BB}$ intersects the closure of our 1-dimensional cusp $s$ precisely at this 0-dimensional cusp, but does not intersect the open cusp $s$ at all. Thus if $\gamma J \not\subseteq V_{\ul{x}}$ for any $\gamma\in\Gamma$, then the closure of $s$ in $X^{\BB}$ will intersect $\iota_{\ul{x}}^{\BB}(X_{\ul{x}}^{\BB})$ at the images of the 0-dimensional cusps $s_{0}$ associated with such subspaces $I$, but none of these intersection points will lie inside $s$ itself.

\medskip

It is clear that if the space $V_{\ul{x}}$ is indefinite and anisotropic, then $\mathcal{D}_{\ul{x}}$ has positive dimension and no rational boundary components, and $X_{\ul{x}}$ is compact. It is thus equal to all of its compactifications such as $X_{\ul{x}}^{\BB}$ and $X_{\ul{x}}^{\tor}$, and its image under $\iota_{\overline{x}}$ is a compact subset of $X$. We note that in Case 2 this can only happen when $r(\ul{x})=n-1$.

In Case 1, Lemma \ref{cuspsvecx} also yields the following immediate consequence.
\begin{cor} \label{Witt1}
When $V_{\ul{x}}$ has Witt index 1, each cusp of $X_{\ul{x}}^{\BB}$ maps to a 0-dimensional cusp of $X^{\BB}$. In particular, this happens wherever $X_{\ul{x}}$ is a non-compact curve.
\end{cor}
Corollary \ref{Witt1} becomes clear once one observes that the isotropic subspaces of $V_{\ul{x}}$ are all 1-dimensional, and that $X_{\ul{x}}$ is a non-compact curve precisely when $V_{\ul{x}}$ is isotropic of signature $(1,2)$, hence of Witt index 1.

\medskip

\section{Hermitian Modularity \label{HermMod}}

We will first treat Case 2. For $\tau$ in the Hermitian upper half-space $\mathbb{H}_{n}$, recall that the formal generating series from \eqref{unifgenser} takes the form
\begin{equation} \label{genserHerm}
F^{\circ}(\tau)=\sum_{\ul{\mu}\in(L^{*}/L)^{n}}\sum_{T\in\operatorname{Herm}_{n}(\mathbb{K})_{\geq0}}Z^{\circ}(T,\ul{\mu})\mathbf{e}\big(\operatorname{tr}(T\tau)\big)\mathfrak{e}_{\ul{\mu}}.
\end{equation}
We aim to show that this series is a Hermitian modular form with coefficients in $\CH^{n}(X^{\tor})_{\CC}$. We will do so by viewing this series as the image of the generating series
\begin{equation} \label{Hermserop}
F(\tau)=\sum_{\ul{\mu}\in(L^{*}/L)^{n}}\sum_{T\in\operatorname{Herm}_{n}(\mathbb{K})_{\geq0}}Z(T,\ul{\mu})\mathbf{e}\big(\operatorname{tr}(T\tau)\big)\mathfrak{e}_{\ul{\mu}}\in \CH^{n}(X)_{\CC},
\end{equation}
of the special cycles from \eqref{spcycop} on the open variety $X$, under the isomorphism from Proposition \ref{isom0cyc} below.

Recall that $X$ is the unitary (open) connected Shimura variety associated with a Hermitian space of signature $(n,1)$ as in our Case 2, and $X^{\tor}$ is its canonical toroidal compactification. Recall from \cite{M} that the class $[s]$ of a point in the boundary is well-defined (see Theorem \ref{pt0cusp}), and let $\CH^{n}(X^{\tor})^{0}_{\QQ}$ denote the kernel of the degree map $\deg:\CH^{n}(X^{\tor})_{\QQ}\to\QQ$. We let $Z^{n}(X)$ denote the group of cycles of codimension $n$ on $X$, and similarly for $X^{\tor}$, and prove the following result.
\begin{prop} \label{isom0cyc}
For $s$ as above, the map of cycles
\[
\tilde{\psi}:Z^{n}(X)_{\QQ} \to Z^{n}(X^{\tor})_{\QQ},\quad\tilde{\psi}(y):=y-\deg(y) \cdot s
\]
induces an isomorphism of groups
\[\psi:\CH^{n}(X)_{\QQ}\to\CH^{n}(X^{\tor})^{0}_{\QQ}.
\]
\end{prop}

\begin{proof}
The result will follow from the localization sequence for Chow groups. Let $Y=X^{\tor} \setminus X$ be the boundary of $X^{\tor}$ viewed as a closed subvariety of codimension 1. Let $j:Y \to X^{\tor}$ be the corresponding closed immersion, and let $i:X \to X^{\tor}$ be the associated open immersion. Then the codimension 1 property implies that for every $1 \leq k \leq n$, the localization sequence of Chow groups is the exact sequence
\[
\CH^{k-1}(Y)\stackrel{j_{*}}{\to}\CH^{k}(X^{\tor})\stackrel{i^{*}}{\to}\CH^{k}(X)\to0.
\]
For $k=n=\dim(X)$, the restriction of the pullback $i^{*}$ to the degree 0 part of $\CH^{n}(X^{\tor})^{0}_{\QQ}$ yields a map
\[
\varphi:\CH^{n}(X^{\tor})^{0}_{\QQ}\to\CH^{n}(X)_{\QQ},
\]
which is still surjective as the boundary $Y$ is non-empty. We will show that our asserted map $\psi$ is a well-defined inverse for $\varphi$, which will then automatically be an isomorphism.

For $s$ fixed, we claim that our map $\tilde{\psi}$ on 0-cycles is compatible with rational equivalence. To see this, let $D$ be a cycle on $X$ that is rationally equivalent to 0. Without loss of generality, we may assume that
\[
D=\iota_{*}\big(\operatorname{div}_{C}(f)\big)
\]
for a (quasi-projective) curve $C \subseteq X$ with embedding $\iota:C \to X$, and a rational function $f\in\QQ(C)^{\times}$. We need to show that $\tilde{\psi}(D)$ is equivalent to a divisor of a rational function on a curve on $X^{\tor}$.

But the closure $\overline{C}$ of $C$ in $X^{\tor}$ is a projective curve on $X^{\tor}$, embedded into it via the map $\overline{\iota}$, and the pushforward of the divisor of $f$ on $\overline{C}$ defines a degree 0 cycle
\[
\overline{\iota}_{*}\big(\operatorname{div}_{\overline{C}}(f)\big)=D+B,
\]
where $B$ is supported on $Y$ (since it contains the cusps of $\overline{C}$). But by Theorem \ref{pt0cusp}, $B$ is rationally equivalent to $\deg B \cdot s$, which is the same as $-\deg D \cdot s$ by the degree 0 condition. Thus the right hand side is rationally equivalent to $\tilde{\psi}(D)$, and as the left hand side is rationally equivalent to 0 by definition, our claim is proved.

It follows that $\tilde{\psi}$ induces a well-defined map $\psi$ from $\CH^{n}(X)_{\QQ}$ to $\CH^{n}(X^{\tor})_{\QQ}$, with image that is contained in $\CH^{n}(X^{\tor})^{0}_{\QQ}$. As it is easily verified that composing it with $\varphi$ from either direction produces the identity maps, $\psi$ and $\varphi$ are inverses as asserted. This completes the proof of the proposition.
\end{proof}

We note that the image of each of the toroidal special cycles $Z^{\circ}(T,\ul{\mu})$ under the pullback map $\varphi:\CH^{n}(X^{\tor})^{0}_{\QQ}\to\CH^{n}(X)_{\QQ}$ is given by the usual special cycle $Z(T,\ul{\mu})$ on the open Shimura variety $X$, namely we have
\[
\varphi\big(Z^{\circ}(T,\ul{\mu})\big)=Z(T,\ul{\mu})
\]
using the map $\varphi$ from the proof of Proposition \ref{isom0cyc}. As $\varphi$ is an isomorphism with inverse $\psi$ by that proposition, it follows immediately that the generating series $F^{\circ}(\tau)$ is the image of the generating series $F(\tau)$ from \eqref{Hermserop} under the map $\psi$, namely we have
\[
F^{\circ}(\tau)
=
\psi\big(F(\tau)\big)\in\CH^{n}(X^{\tor})_{\QQ}\ \ \ \text{for all }\tau\in\mathbb{H}_{n}.
\]

Now, Theorem 6.1 of \cite{Xi} states that the series $F$ from \eqref{Hermserop} is a Hermitian modular form when the imaginary quadratic field $\mathbb{K}$ is norm-Euclidean (namely its discriminant is $-3$, $-4$, $-7$, $-8$, or $-11$). Moreover, Theorem 3.5 of \cite{Liu} implies that if the series $F$ converges absolutely, then it is a Hermitian modular form without the restriction on the discriminant of $\mathbb{K}$. As applying the isomorphism $\psi$ from Proposition \ref{isom0cyc} preserves modularity, the same assertions hold for the series $F^{\circ}$ from \eqref{genserHerm} as well. This proves Theorems \ref{modunit} and \ref{modafterconv} from the Introduction.

\if{
The modularity follows immediately from the modularity of $F(\tau)$. Recalling the unconditional modularity of the generating series $F(\tau)$ from \cite{Xi}, as well as the conditional modularity of \cite{Liu}, we obtain the following two results.
\begin{thm} \label{thmHerm}
For a norm-Euclidean quadratic imaginary field $\mathbb{K}$, the generating series $F^{\circ}$ is a Hermitian modular form of degree $n$, weight $1+n$, and representation $\rho_{L,n}$ with respect to $\operatorname{U}_{n,n}(\ZZ)$, with values in the Chow group $\CH^{n}(X^{\tor})_{\CC}$.
\end{thm}

\begin{thm} \label{modconv}
Assume that the generating series $F^{\circ}(\tau)$ converges absolutely. Then $F^{\circ}$ is a Hermitian modular form of degree $n$, weight $1+n$ and representation $\rho_{L,n}$ with respect to $\operatorname{U}_{n,n}(\ZZ)$, valued in the Chow group $\CH^{n}(X^{\tor})_{\CC}$.
\end{thm}
This establishes Theorems \ref{modunit} and \ref{modafterconv} from the Introduction.
} \fi

\section{Siegel Modularity \label{SiegelMod}}

The goal of this section is to prove that the generating series from \eqref{unifgenser} is, in Case 1, a Siegel modular form with coefficients in $\CH^{n}(X^{\tor})_{\CC}$.

\subsection{Special Chow Groups}

Here we will work with a variant of the Chow group which we shall call the \emph{special Chow group}. We will only define it here for the maximal codimension $n$. Like the usual Chow group, it is defined as the quotient of 0-cycles by appropriate rational equivalence.

As before, let $Z^{n}(X)$ (resp.\ $Z^{n}(X^{\tor})$) denote the group of cycles of codimension $n$ on the open Shimura variety $X$ (resp.\ the toroidal compactification $X^{\tor}$). If $i:X \hookrightarrow X^{\tor}$ is the natural open immersion, then it defines an injective pushforward map $i_{*}:Z^{n}(X) \hookrightarrow Z^{n}(X^{\tor})$.

On $X$, we define the group of \emph{special 0-cycles} to be
\[
Z^{n}_{\operatorname{sp}}(X):=\Span_{\ZZ}\{y \in X\;|\;y\text{ a special cycle of dimension }0\} \subseteq Z^{n}(X).
\]
After we identify $Z^{n}_{\operatorname{sp}}(X)$ with its $i_{*}$-image in $Z^{n}(X^{\tor})$, we define the group $Z^{n}_{\operatorname{sp}}(X^{\tor})$ of \emph{special 0-cycles on $X^{\tor}$} to be the direct sum of $Z^{n}_{\operatorname{sp}}(X)$ and the group
\begin{equation} \label{Bn0Xtor}
B^{n}_{0}(X^{\tor}):=\Span_{\ZZ}\{s \in X^{\tor}\;|\;\pi(s) \in X^{\BB}\text{ is a cusp of dimension 0}\}.
\end{equation}

For making some definitions and claims for both types of varieties simultaneously, we shall henceforth write here $\mathbb{X}$ for either the open Shimura variety $X$ or for its toroidal compactification $X^{\tor}$. Let $y$ and $\tilde{y}$ be elements of $Z^{n}_{\operatorname{sp}}(\mathbb{X})$. We say that they are equivalent in the relation of \emph{special rational equivalence} if
\[
y-\tilde{y}=\iota_{*}(\operatorname{div}_{W}f),
\]
where $\iota:W\hookrightarrow\mathbb{X}$ is an orthogonal Shimura subvariety of dimension 1 and $f\in\QQ(W)^{\times}$ is a rational function on $W$ whose divisor is supported on the special divisors on $W$, as well as on the boundary points (i.e., cusps) of $W$ in case $\mathbb{X}=X^{\tor}$. Note that the result of \cite{M} cited here as Theorem \ref{pt0cusp} implies that two cycles that are supported only on $B^{n}_{0}(X^{\tor})$ coincide modulo special rational equivalence if and only if they have the same degree.

We denote the equivalence relations coming from special rational equivalence by $\sim_{\operatorname{sp}}$, and the usual rational equivalence by $\sim_{\operatorname{rat}}$. We then define the quotients
\[
\SCH^{n}(\mathbb{X}):=Z^{n}_{\operatorname{sp}}(\mathbb{X})/\sim_{\operatorname{sp}}\qquad\mathrm{and}\qquad\SCH'^{n}(\mathbb{X}):=Z^{n}_{\operatorname{sp}}(\mathbb{X})/\sim_{\operatorname{rat}}.
\]
the first of which we call the $n$th \emph{special Chow group} of $\mathbb{X}$ (the second one is the image of $Z^{n}_{\operatorname{sp}}(\mathbb{X})$ in $\CH^{n}(\mathbb{X})$).

\subsection{Pullbacks and Isomorphisms of Special Chow Groups}

We observe that our rational equivalence relation $\sim_{\operatorname{sp}}$ is finer than the usual one $\sim_{\operatorname{rat}}$. Indeed, this is obvious when $\mathbb{X}=X$, and follows from Theorem \ref{pt0cusp} in case $\mathbb{X}=X^{\tor}$. We thus have a projection and an injection, which we write as
\begin{equation} \label{SCHgrps}
\SCH^{n}(\mathbb{X})\twoheadrightarrow\SCH'^{n}(\mathbb{X})\hookrightarrow\CH^{n}(\mathbb{X}).
\end{equation}
Moreover, we can compose this projection and this injection with the usual cycle class map from $\CH^{n}(\mathbb{X})_{\QQ}$ into the cohomology group $H^{2n}(\mathbb{X},\QQ)$, in order to obtain the class map
\[
\operatorname{cl}:\SCH^{n}(\mathbb{X})_{\QQ} \to H^{2n}(\mathbb{X},\QQ).
\]
The kernel of the class map is the space of classes of cohomologically trivial special cycles in $\SCH^{n}(\mathbb{X})$, namely
\[
\SCH^{n}(\mathbb{X})_{\QQ}^{0}:=\{y\in\SCH^{n}(\mathbb{X})_{\QQ}\;|\;\operatorname{cl}(y)=0 \in H^{2n}(\mathbb{X},\QQ)\}.
\]

Now, the open immersion $i:X \hookrightarrow X^{\tor}$ induces the pullback map \[i^{*}:Z^{n}_{\operatorname{sp}}(X^{\tor}) \to Z^{n}_{\operatorname{sp}}(X)\] on cycles, and as with the usual pullback on the Chow groups, we obtain another type of pullback map.
\begin{lem} \label{pbSCH}
The pullback map $i^{*}$ induces a well-defined surjective map of special Chow groups \[i^{*}:\SCH^{n}(X^{\tor})\to\SCH^{n}(X).\]
\end{lem}

\begin{proof}
The fact that $Z^{n}_{\operatorname{sp}}(X)$ is a subgroup of $Z^{n}_{\operatorname{sp}}(X^{\tor})$ and also maps surjectively onto $\SCH^{n}(X)$ means that once $i^{*}$ is shown to be well-defined, it is immediately surjective.

Thus all that we need to show is that if an element $y \in Z^{n}_{\operatorname{sp}}(X^{\tor})$ becomes trivial in $\SCH^{n}(X^{\tor})$, then its $i^{*}$-image becomes trivial in $\SCH^{n}(X)$. But our assumption implies the existence of a 1-dimensional special subvariety $W$ of $X^{\tor}$, embedded via the map $\iota$, and a function $f\in\QQ(W)^{\times}$ whose divisor is supported on the special divisors and the cusps of $W$, such that
\[
y=\iota_{*}(\operatorname{div}_{W}f).
\]
But then $W \cap X$ is an open curve that is embedded as a closed subvariety of $X$, the divisor of $f$ on $W \cap X$ is supported only on the special divisors on that curve, and the cycle $i^{*}y=i^{*}\iota_{*}(\operatorname{div}f)$ coincides with $\big(\iota|_{W \cap X}\big)_{*}\big(\operatorname{div}(f|_{W \cap X})\big)$. Thus the image of $i^{*}y$ vanishes in $\SCH^{n}(X)$, as desired. This proves the lemma.
\end{proof}
Now, we are working under the assumption that the boundary $X^{\tor} \setminus X$ of $X^{\tor}$ is non-empty. We can thus restrict the map from Lemma \ref{pbSCH} to elements of degree 0 and still obtain a surjective map
\begin{align} \label{invSCH}
\varphi:\SCH^{n}(X^{\tor})^{0}\to\SCH^{n}(X).
\end{align}

\smallskip

We can now prove the following analogue of Proposition \ref{isom0cyc}, in which again $s$ is any boundary point on $X^{\tor}$ whose $\pi$-image is a cusp of dimension 0 on $X^{\BB}$.
\begin{prop} \label{cyc0isom}
The map
\[
\tilde{\psi}:Z^{n}_{\operatorname{sp}}(X) \to Z^{n}_{\operatorname{sp}}(X^{\tor}),\quad\tilde{\psi}(y):=y-\deg(y) \cdot s
\] induces an isomorphism
\[
\psi:\SCH^{n}(X)_{\QQ}\stackrel{\sim}{\to}\SCH^{n}(X^{\tor})_{\QQ}^{0}.
\]
\end{prop}

\begin{proof}
We need to show that $\tilde{\psi}$ is compatible with the special rational equivalence, meaning that if $y \in Z^{n}_{\operatorname{sp}}(X)$ satisfies $y\sim_{\operatorname{sp}}0$, then $\tilde{\psi}(y)\sim_{\operatorname{sp}}0$ as well.

Now, our assumption on $y$ implies the existence of a special Shimura variety $W$ of dimension 1, with embedding $\iota:W \hookrightarrow X$, and a rational function $f\in\QQ(W)^{\times}$ whose divisor is supported on the special divisors on $W$, such that
\[
y=\iota_{*}(\operatorname{div}_{W}f).
\]

But then the closure $\overline{W}$ of $W$ in $X^{\tor}$ is a special Shimura subvariety of $X^{\tor}$, with embedding $\overline{\iota}:\overline{W} \hookrightarrow X^{\tor}$, and we can view $f$ as an element of $\QQ(\overline{W})^{\times}$. It follows that
\[
\overline{\iota}_{*}\operatorname{div}_{\bar{C}}(f)=D+B,
\]
where $D$ is supported on special cycles of $X$ and $B$ is supported in the boundary $X^{\tor} \setminus X$. Moreover, we have $D=y$ by restricting everything to $X$.

If $W$ is not compact, then Corollary \ref{Witt1} (or Lemma \ref{intcusp}) shows that the boundary points of $\overline{C}$ get sent, by the pushforward $\iota_{*}:Z^{1}(\overline{W}) \hookrightarrow Z^{n}(X^{\tor})$, to elements of $B^{n}_{0}(X^{\tor})$ from \eqref{Bn0Xtor}. This implies, by Theorem \ref{pt0cusp}, that $B\sim_{\operatorname{sp}}\deg B \cdot s$ for any fixed point $s \in X^{\tor}$ lying over a 0-dimensional cusp of $X^{\BB}$. Recalling that $D=y$ and that $D+B$ is of degree 0 (as the pushforward of the divisor of a rational function on a complete curve), we deduce via the definition of $\sim_{\operatorname{sp}}$ that
\[
\tilde{\psi}(y)\sim_{\operatorname{sp}}=y-\deg(y) \cdot s=D+\deg B \cdot s\sim_{\operatorname{sp}}D+B=\overline{\iota}_{*}\operatorname{div}_{\bar{C}}(f)\sim_{\operatorname{sp}}0,
\]
as desired. Hence the induced map $\psi:\SCH^{n}(X)_{\QQ}\to\SCH^{n}(X^{\tor})^{0}_{\QQ}$ is well-defined.

Finally, one easily verifies that the compositions of $\psi$ with the pullback map $\varphi$ from \eqref{invSCH} in both directions yield the respective identity maps, so that they are inverses and $\psi$ is an isomorphism as asserted. This proves the proposition.
\end{proof}

\subsection{Modularity}

We recall the generating series of special points on $X$ is defined to be
\begin{equation} \label{BRgenser}
F(\tau)=\sum_{\ul{\mu}\in(L^{*}/L)^{n}}\sum_{T\in\Sym_{n}(\QQ)_{\geq0}}Z(T,\ul{\mu})\mathbf{e}\big(\operatorname{tr}(T\tau)\big)\mathfrak{e}_{\ul{\mu}},
\end{equation}
and that \cite{BR} proves it to be a Siegel modular form of degree $n$, weight $1+\frac{n}{2}$, and representation $\rho_{L,n}$ with respect to $\widetilde{\Sp}_{2n}(\ZZ)$, with values in $\CH^{n}(X)_{\CC}$.

Moreover, we note that the rational equivalence relations between the special cycles on $X$ required in the argument of \cite{Zh} all come from Borcherds products and therefore only involve special cycles. It follows that the arguments of \cite{BR} and \cite{Zh} can be adjusted for establishing the following theorem.
\begin{thm} \label{BRSCH}
For $\tau\in\HH_{n}$, the generating series $F(\tau)$ is a Siegel modular form of degree $n$, weight $1+\frac{n}{2}$, and representation $\rho_{L,n}$ with respect to $\widetilde{\Sp}_{2n}(\ZZ)$, with values in $\SCH^{n}(X)_{\CC}$.
\end{thm}

We are now ready to establish the first version of the modularity of the series from \eqref{unifgenser}, which here takes the form
\begin{equation} \label{fingenser}
F^{\circ}(\tau):=\sum_{\ul{\mu}\in(L^{*}/L)^{n}}\sum_{T\in\Sym_{n}(\QQ)_{\geq0}}Z^{\circ}(T,\ul{\mu})\mathbf{e}\big(\operatorname{tr}(T\tau)\big)\mathfrak{e}_{\ul{\mu}}, \ \ \tau\in\HH_{n}.
\end{equation}

As in the previous case, we note that the class in $\SCH^{n}(X^{\tor})^{0}$ of any toroidal special 0-cycle $Z^{\circ}(T,\ul{\mu})$ maps via the pullback map $\varphi$ from \eqref{invSCH} to the usual special 0-cycles $Z(T,\ul{\mu})$, as an element of $\SCH^{n}(X)$. Then, as Proposition \eqref{cyc0isom} shows that the map $\varphi$ is an isomorphism with inverse $\psi$, it follows that for each $\tau\in\mathcal{H}_{n}$ we have the equality
\[
F^{\circ}(\tau)=\psi(F(\tau))
\]
of the series from \eqref{BRgenser} and \eqref{fingenser}, as series with values in the asserted special Chow groups. After embedding the subspace $\SCH^{n}(X^{\tor})_{\CC}^{0}$ into $\SCH^{n}(X^{\tor})_{\CC}$, we can view the series $F^{\circ}(\tau)$ as valued in $\SCH^{n}(X^{\tor})_{\CC}$. Thus the modularity of $F$ given in Theorem \ref{BRSCH} combines with the latter equality to establish the following theorem.
\begin{thm} \label{mainthmSCH}
The generating series $F^{\circ}(\tau)$ is a Siegel modular form of degree $n$, weight $1+\frac{n}{2}$, and representation $\rho_{L,n}$ with respect to $\widetilde{\Sp}_{2n}(\ZZ)$, with coefficients in the special Chow group $\SCH^{n}(X^{\tor})_{\CC}$.
\end{thm}

Finally, the natural maps $\SCH^{n}(X^{\tor})_{\CC} \twoheadrightarrow \SCH'^{n}(X^{\tor})_{\CC} \hookrightarrow\CH^{n}(X^{\tor})_{\CC}$ from \eqref{SCHgrps} preserve the modularity from Theorem \ref{mainthmSCH}, thus we get the modularity of the generating series in the Chow group as well, as stated in Theorem \ref{modorth} in the Introduction.

\appendix

\section{Boundary Behavior \label{BoundBeh}}

A technical point that is essential for our proof is that the boundary of our 0-dimensional toroidal cycles will only involve points that map to 0-dimensional cusps under the canonical map $\pi:X^{\tor} \to X^{\BB}$ from \eqref{Xtornot}. This statement is vacuous when $V$ is anisotropic and $X$ is compact, and obvious in Case 2 since all the cusps of $X^{\BB}$ are then 0-dimensional. We thus assume that we are in Case 1 and $V$ is isotropic (which is always the case when $n\geq3$), and thus, in particular, 0-dimensional cusps exist.

For establishing this result, we begin with the following lemma.
\begin{lem} \label{intover0dim}
For $1 \leq j \leq n$, we take $Z_{j}$ to be an irreducible divisor on $X^{\tor}$ such that either $Z_{j}=\overline{Z(y_{j})}$ is the closure of a special divisor for some $y_{j} \in V$ with $\QQ y_{j}$ positive definite, or $Z_{j}$ is a boundary divisor. We assume that, for each $i<n$, no irreducible component of $\bigcap_{j=1}^{i}Z_{j}$ is contained in $Z_{i+1}$. Then $\bigcap_{j=1}^{n}Z_{j}$ is, when non-empty, a collection of points, whose image under $\pi$ is supported:
\begin{itemize}
\item on $X$, if none of the $Z_{j}$'s is a boundary divisor.
\item on the preimage in $X^{\tor}$ of the 0-dimensional cusps of $X^{\BB}$, otherwise.	
\end{itemize}	
\end{lem}

\begin{proof}
The condition about the intersection not being contained in the next divisor implies that $\bigcap_{j=1}^{n}Z_{j}$ is either empty or 0-dimensional, hence a collection of points. Moreover, the result is clear if one of the divisors $Z_{j}$ is a boundary divisor $B_{I,\omega}$, or if two divisors lying over 1-dimensional cusps intersect, as their intersection in $X^{\tor}$ lies over 0-dimensional cusps of $X^{\BB}$.

We can thus assume, without loss of generality, that for any $1 \leq j \leq n-1$ we have
\[
Z_{j}=\overline{Z(y_{j})}.
\]
Thus their intersection $\bigcap_{j=1}^{n-1}\overline{Z(y_{j})}$ is a finite collection of curves of the form $\overline{Z(\ul{x})}$ for $\ul{x} \in V^{n-1}$ with $\QQ\ul{x}$ positive definite space of dimension $n-1$. Every such $\overline{Z(\ul{x})}$ is the closure in $X^{\tor}$ of the corresponding special cycle $Z(\ul{x})$ on $X$, and, using the pushforward of \eqref{iotator}, it equals
\[
\overline{Z(\ul{x})}=\iota_{\ul{x}, *}^{\tor}(X^{\tor}_{\ul{x}}).
\]
Then Corollary \ref{Witt1} implies that $\overline{Z(\ul{x})}$ intersects the boundary of $X^{\tor}$ only in points lying over 0-dimensional cusps. Thus, if $Z_{n}=\overline{Z(y_{n})}$ is a special divisor, then $Z(\ul{x})$ intersects with $Z_{n}$ either in points of $X$ or in points lying over 0-dimensional cusps. Otherwise, if $Z_{n}$ is a boundary divisor, Corollary \ref{Witt1} again yields that the intersection of $Z(\ul{x})$ with the boundary of $X^{\tor}$ is supported, when not empty, on the preimages of the 0-dimensional cusps of $X^{\BB}$. This proves the lemma.
\end{proof}

\subsection{Intersection with the Dual Tautological Bundle}

Let again $\LL_{\tor}$ denote the tautological line bundle on $X^{\tor}$ (in both Cases 1 and 2), with dual line bundle $\LL_{\tor}^{\vee}$ and corresponding Chow class $[\LL_{\tor}^{\vee}]$. We will need the intersections, in the Chow ring of $X^{\tor}$, of any special cycle class $[\overline{Z(\ul{x}})]$, and of the class of the pushforward of any of the boundary divisors of $X_{\ul{x}}^{\tor}$, with the power of $[\LL_{\tor}^{\vee}]$ such that the product is in $\CH^{n}(X^{\tor})$.

\begin{prop} \label{dimbd}
Given an element $\ul{x} \in V^{r}$, the class $[\overline{Z(\ul{x})}]\cdot[\LL_{\tor}^{\vee}]^{n-r(\ul{x})}$ can be represented in $\CH^{n}(X^{\tor})$ by a 0-cycle that is supported on $X$.
\end{prop}

\begin{proof}
Take a basis $\{y_{j}\}_{1 \leq j \leq r(\ul{x})}$ for the vector space $\ul{x}\QQ$ and let $Z_{j}:=\overline{Z(y_{j})}$ be the corresponding special divisors. We can choose an orthonormal basis such that for each $i$, the divisor $Z_{i+1}$ does not contain any irreducible component of $\bigcap_{j=1}^{i}Z_{j}$. Then it is clear that if we prove the assertion for the intersection of the class of $\bigcap_{j=1}^{r(\ul{x})}Z_{j}$ with $n-r(\ul{x})$ copies of $\LL_{\tor}^{\vee}$, then the statement for $\overline{Z(\ul{x})}$ will follow.

Assume that we are in Case 1. The sections of the line bundle $\LL_{\tor}^{\vee}$ are modular forms of weight $-\frac{1}{2}$ on $X^{\tor}$. We can define such sections using Borcherds products as follows (see \cite{Bo1} for the general construction of these functions). The $j^{\text{th}}$ section of $\LL_{\tor}^{\vee}$ is constructed from a weakly holomorphic modular form
\[
f_{j}(z):=\sum_{\mu \in L^{*}/L}\sum_{m\gg0}c_{j}(m,\mu)\mathbf{e}(mz)e_{\mu},
\]
with $c_{j}(m,\mu)\in\ZZ$ for every $m$ and $\mu$. The divisor of the associated Borcherds product $\Psi_{j}$ on $X^{\tor}$ is determined in Theorem 5.2 of \cite{BZ} to be a combination of special divisors and boundary divisors, namely
\[
\operatorname{div}_{X^{\tor}}(\Psi_{j})
=
\frac{1}{2}\sum_{\mu \in L^{*}/L}\sum_{\substack{m\in\ZZ+Q(\mu) \\ m>0}}c_{j}(-m,\mu)Z^{\tor}(m,\mu),
\]
where $Z^{\tor}(m,\mu)$ is the toroidal special divisor defined by \cite{BZ}. Moreover, by allowing the pole of $f_{j}$ to be of high enough order, we can obtain a linearly independent set of modular forms with integral Fourier coefficients of arbitrarily large cardinality. We use this family to construct Borcherds products of any finite cardinality that we wish.

In Case 2, the sections of the line bundle $\LL_{\tor}^{\vee}$ are modular forms of weight $-1$. In \cite{Ho2}, Hofmann explains how to construct such modular forms as Borcherds products in this setting.

We are interested in the intersection
\[
\bigcap_{j=1}^{r(\ul{x})}Z_{j}\cdot \operatorname{div}\Psi_{1}\cdot\ldots\cdot\operatorname{div}\Psi_{r(\ul{x})}.
\]
In both cases it is known that, for Borcherds products, the total multiplicity of each boundary divisor $B$ is rational (see \cite{BZ} for Case 1 and \cite{BHKRY} for Case 2). As $X$ carries finitely many such divisors, the condition that the boundary divisors will not appear in $\operatorname{div}_{X^{\tor}}(\Psi)$ gives a bounded number of linear restrictions over $\mathbb{Z}$ on the weakly holomorphic modular forms $f$. Moreover, given any special cycle $Z(\ul{y})$ on $X$, we can assume, by having an appropriate combination of enough such modular forms, that the divisor of $\Psi_{j}$ does not contain $\overline{Z(\ul{y})}$. This is also valid for finitely many distinct such special divisors.

Thus we can choose each successive divisor $\operatorname{div}\Psi_{j}$ with restrictions as above, for $1\leq j\leq r(\ul{x})$. As $\bigcap_{j=1}^{r(\ul{x})}Z_{j}$ is the union of finitely many special cycles $Z(\ul{y}^{(i)})$, we can take $\operatorname{div}\Psi_{1}$ such that its support
\begin{itemize}
\item does not contain any boundary divisors, and
\item does not contain any special cycle $\overline{Z(\ul{y}^{(i)})}$, $1\leq i \leq r(\ul{x})$.
\end{itemize}
Since $\operatorname{div}\Psi_{1}$ is then a linear combination of special divisors, the intersection with each $\overline{Z(\ul{y}^{(i)})}$ will be a finite combination of special divisors of codimension 1 more. But then for $\Psi_{2}$ we are in a similar situation, and we do so until we obtain, by Lemma \ref{intover0dim}, an intersection which is represented by a finite combination of special points (i.e., special cycles of dimension 0), which lie on $X$. This proves the proposition.
\end{proof}
It is clear from the proof of Proposition \ref{dimbd} that every such representation is a special 0-cycle in the terminology from Section \ref{SiegelMod}, and that they all represent the same class in the special Chow group $\SCH^{n}(X^{\tor})$ as well.

\medskip

For the boundary intersections, we shall be using the following result of \cite{M}.
\begin{thm} \label{pt0cusp}
If $s$ and $t$ are points in $X^{\tor}$ that map to 0-dimensional cusps on $X^{\BB}$, then the points $s$ and $t$ represent the same class in $\CH^{n}(X^{\tor})_{\QQ}$.
\end{thm}
We can thus choose any point $s \in X^{\tor}$ such that $\pi(s) \in X^{\BB}$ is a 0-dimensional cusp and denote by $[s]$ the image of $s$ in $\CH^{n}(X^{\tor})_{\QQ}$, and then Theorem \ref{pt0cusp} implies that this class is canonical, i.e., independent of the choice of $s$. In particular, the group $B^{n}_{0}(X^{\tor})$ from \eqref{Bn0Xtor} has rank 1.

We can now obtain the following result.
\begin{lem} \label{intcusp}
Let $B_{\ul{x}}$ be a boundary divisor on $X_{\ul{x}}^{\tor}$, and let $\iota_{\ul{x},*}^{\tor}$ be the pushforward map from \eqref{pushforward}. Then the class of the $n$-cycle $\iota_{\ul{x},*}^{\tor}[B_{\ul{x}}]\cdot[\LL_{\tor}^{\vee}]^{n-r(\ul{x})-1}$ in $\CH^{n}(X^{\tor})_{\QQ}$ can be represented by a multiple of $[s]$.
\end{lem}

\begin{proof}
There is a divisor $B$ on $X^{\tor}$ such that its intersection with $\overline{Z(\ul{x})}$ is the sum of $\iota_{\ul{x},*}^{\tor}$-images of boundary divisors on $X_{\ul{x}}^{\tor}$, including $B_{\ul{x}}$. As in the proof of Proposition \ref{dimbd}, we can consider $Z(\ul{x})$ as an irreducible component of an intersection $\bigcap_{j=1}^{r(\ul{x})}Z_{j}$ of special divisors on $X^{\tor}$, such that the other components are also of dimension $n-r(\ul{x})$. Again establishing the result for the intersection of the $Z_{j}$'s, $B$, and $n-r(\ul{x})-1$ copies of $\LL_{\tor}^{\vee}$ would imply the desired assertion for $\iota_{\ul{x},*}^{\tor}[B_{\ul{x}}]\cdot[\LL_{\tor}^{\vee}]^{n-r(\ul{x})-1}$.

Now, the proof of Proposition \ref{dimbd} also shows that the divisors of appropriately chosen sections of $\LL_{\tor}^{\vee}$ are sums of special cycles, such that adding each one of them successively to the $Z_{j}$'s, and setting $Z_{n}:=B$, produces a sequence satisfying the condition from Lemma \ref{intover0dim}. That lemma shows that in Case 1, this intersection is a collection of points having $\pi$-images that are 0-dimensional cusps of $X^{\tor}$. In Case 2, the intersection is also a collection of points, and as all of them lie on $B$, their $\pi$-images are all the 0-dimensional cusp $\pi(B) \subseteq X^{\tor}$. As Theorem \ref{pt0cusp} implies that in both cases, all these points give rise to the same class $[s]$ in $\CH^{n}(X^{\tor})_{\QQ}$, the result follows. This proves the lemma.
\end{proof}

We deduce the following consequence.
\begin{cor} \label{altbys}
Let $Z$ be any cycle on $X^{\tor}$ of codimension $r+1$, and let $\widetilde{Z}$ be a cycle which is the sum $Z+\sum\limits_{r(\ul{x})=r}a_{\ul{x}}(\iota_{\ul{x},*}^{\tor}B_{\ul{x}})$ of $Z$ and a linear combination of pushforwards of boundary divisors $B_{\ul{x}}$ of the varieties $X_{\ul{x}}^{\tor}$, for $\ul{x}$ such that $r(\ul{x})=r$. Then the element \[[Z]\cdot[\LL_{\tor}^{\vee}]^{n-r(\ul{x})-1}-\deg\big([Z]\cdot[\LL_{\tor}^{\vee}]^{n-r(\ul{x})-1}\big)\cdot[s]\in\CH^{n}(X^{\tor})_{\QQ}\] gives the same class as $[\widetilde{Z}]\cdot[\LL_{\tor}^{\vee}]^{n-r(\ul{x})-1}-\deg\big([\widetilde{Z}]\cdot[\LL_{\tor}^{\vee}]^{n-r(\ul{x})-1}\big)\cdot[s]$.
\end{cor}

\begin{proof}
By Lemma \ref{intcusp}, the difference between the cycles $[Z]\cdot[\LL_{\tor}^{\vee}]^{n-r(\ul{x})-1}$ and $[\widetilde{Z}]\cdot[\LL_{\tor}^{\vee}]^{n-r(\ul{x})-1}$ is a multiple of $[s]$. As $\deg[s]=1$, the degree gives the multiple, and the desired equality follows. This proves the corollary.
\end{proof}

\section{Non-Triviality in the Chow Group \label{NonTriv}}

Finally, we make some remarks on the non-triviality of the generating series $F^{\circ}(\tau)$ from \eqref{unifgenser}, valued in $\CH^{n}(X^{\tor})_{\CC}$. For each $0 \leq k \leq n$ we denote by $\CH^{k}(X^{\tor})^{0}_{\CC}$ the subspace of cohomologically trivial classes in $\CH^{k}(X^{\tor})_{\CC}$. From the definition in \eqref{spcyctor}, the special cycles $Z^{\circ}(T,\ul{\mu})$ are all cohomologically trivial, thus in $\CH^{n}(X^{\tor})^{0}_{\CC}$.

Now, the Beilinson--Bloch conjecture predicts that the group $\CH^{n}(X^{\tor})^{0}_{\CC}$ lies in perfect duality with $\CH^{1}(X^{\tor})^{0}_{\CC}$ (see Conjecture II on page 32 of \cite{Sch}). In Case 1, when $n$ is larger than the Witt rank $w$ of $V$, the toroidal compactification $X^{\tor}$ carries no holomorphic 1-forms, and therefore, by Hodge theory, the cohomology group $H^{1}(X^{\tor},\CC)$ vanishes. This implies that $\CH^{1}(X^{\tor})^{0}_{\CC}=\{0\}$. Thus, assuming the Beilinson--Bloch conjecture, the value group $\CH^{n}(X^{\tor})^{0}_{\CC}$ of $F^{\circ}(\tau)$ is expected to be trivial. However, no unconditional proof of the vanishing of $\CH^{n}(X^{\tor})^{0}_{\CC}$ seems to be known, even for the classical case of Hilbert modular surfaces.

There are two interesting cases remaining, where the orthogonal Shimura variety $X$ is non-compact and the above conditional argument does not imply the triviality of $\CH^{n}(X^{\tor})^{0}_{\CC}$. First, if $X$ is of dimension 1 and the underlying quadratic space is isotropic, then $X$ is a modular curve and our result essentially specializes to Borcherds' proof of the Gross--Kohnen--Zagier theorem established in \cite{Bo2}. Second, if $X$ is of dimension 2 and the underlying quadratic space has Witt index 2, then $X$ is isogenous to the product of a non-compact modular curve with itself. Note that the structure of $X^{\tor}$ depends on the choice of cone decompositions, and only some specific choices yield the product of a compactified modular curve with itself. In this case $\CH^{1}(X^{\tor})^{0}_{\CC}$ will typically be non-trivial, and thus the Beilinson--Bloch conjecture predicts that $\CH^{2}(X^{\tor})^{0}_{\CC}$ is non-trivial as well. Thus in this case we expect the generating series $F^{\circ}(\tau)$ to be nontrivial.

\medskip

In Case 2, $\CH^{n}(X^{\tor})^{0}_{\CC}$ is non-trivial in general. We present the reasoning for this below, after some notation. Assume that $X^{\tor}$ is smooth, let $\Omega_{X^{\tor}}$ be the sheaf of holomorphic differentials on $X^{\tor}$, and recall the \emph{Albanese variety} of $X^{\tor}$, defined by
\[
\Alb(X^{\tor}):=H^{0}(X^{\tor},\Omega_{X^{\tor}})^{*}/H_{1}(X^{\tor},\ZZ).
\]
Under the assumption that $\CH^{n}(X^{\tor})_{\CC}$ is finite dimensional, Roitman's theorem (see, e.g., Theorem 10.2 of \cite{V}) gives an isomorphism
\[
\CH^{n}(X^{\tor})_{\CC}^{0}\simeq\Alb(X^{\tor}).
\]
Thus it is enough for $\Alb(X^{\tor})$, and thus $H^{0}(X^{\tor},\Omega_{X^{\tor}})$, to be non-trivial. We recall from Section 1 of \cite{MR} that if $n>1$ then $H^{0}(X^{\tor},\Omega_{X^{\tor}})$ is isomorphic to $H^{0}(X,\Omega_{X})$, where $X$ is the open Shimura variety. Then Theorem 8.1 of \cite{Sh} implies that we can find arithmetic subgroups $\Gamma\subseteq\operatorname{U}(V)$ such that $H^{0}(X,\Omega_{X})$ is non-trivial, where $X$ the Shimura variety associated to $\Gamma$. For $n=2$, examples for which $H^{0}(X,\Omega_{X})$ is non-trivial are given in the tables appearing in \cite{Ya}. Furthermore, $H^{0}(X,\Omega_{X})$ is nontrivial for certain arithmetic subgroups also in the case of open unitary Shimura varieties arising from Hermitian spaces over a totally real field that is strictly larger than $\QQ$, as proved in Theorem 1 of \cite{Ka}.


\begin{thebibliography}{BHKRY}

\bibitem[Bo1]{Bo1} Borcherds, R. E., \textsc{Automorphic Forms with Singularities on Grassmannians}, Invent. Math. {\bf 132}, 491--562 (1998).
\bibitem[Bo2]{Bo2} Borcherds, R. E., \textsc{The Gross--Kohnen--Zagier Theorem in Higher Dimensions}, Duke Math J. {\bf 97} no. 2, 219--233 (1999). Correction: Duke Math J. {\bf 105} no. 1, 183--184 (2000).
\bibitem[Br]{Br} Bruinier, J. H., \textsc{Borcherds Products on $\mathrm{O}(2,l)$ and Chern Classes of Heegner Divisors}, Lecture Notes in Mathematics {\bf 1780}, Springer--Verlag (2002).
\bibitem[BHKRY]{BHKRY} Bruinier, J. H., Howard, B., Kudla, S., Rapoport, M., Yang, T., \textsc{Modularity of Generating Series of Divisors on Unitary Shimura Varieties}, Ast\'{e}risque {\bf 421}, 7--125 (2020).
\bibitem[BR]{BR} Bruinier, J. H., Westerholt-Raum, M. \textsc{Kudla's Modularity Conjecture and Formal Fourier--Jacobi Series}, Forum of Math. Pi {\bf 3}, (2015).
\bibitem[BZ]{BZ} Bruinier, J. H., Zemel, S., \textsc{Special Cycles on Toroidal Compactifications of Orthogonal Shimura Varieties}, Math. Ann. {\bf 384}, 1--63 (2022).
\bibitem[EGT]{EGT} Engel, P., Greer, F., Tayou, S., \textsc{Mixed Mock Modularity of Special Divisors}, Preprint, https://arxiv.org/abs/2301.05982 (2023).
\bibitem[Fi]{Fi} Fiori, A., \textsc{Toroidal Compactifications and Dimension Formulas for Spaces of Modular Forms for Orthogonal Shimura Varieties}, pre-print. arXiv link: https://arxiv.org/abs/1610.04865.
\bibitem[Fu]{Fu} Funke, J., \textsc{Heegner Divisors and Nonholomorphic Modular Forms}, Compositio Math. {\bf 133} 289--321 (2002).
\bibitem[Ga]{Ga} Garcia, L., \textsc{Kudla-Millson Forms and One-Variable Degenerations of Hodge Structure}. pre-print, https://arxiv.org/abs/2301.08733
\bibitem[GKZ]{GKZ} Gross, B., Kohnen, W., Zagier, D., \textsc{Heegner Points and Derivatives of $L$-Series, II}, Math. Ann. {\bf 278}, 497--562 (1987).
\bibitem[H]{H} Harris, M., \textsc{Functorial Properties of Toroidal Compactifications of Locally Symmetric Varieties}, Proc. London Math Soc. {\bf 59}, 1--22 (1989).
\bibitem[HZ]{HZ} Hirzebruch, F., Zagier, D., \textsc{Intersection Numbers of Curves on Hilbert Modular Surfaces and Modular Forms of Nebentypus}, Invent. Math. {\bf 36}, 57--114 (1976).
\bibitem[Ho1]{Ho1} Hofmann, E., \textsc{Borcherds Products for $\operatorname{U}(1,1)$}, Int. J. Number Theory {\bf 9}, 1801--1820 (2013).
\bibitem[Ho2]{Ho2} Hofmann, E., \textsc{Borcherds Products on Unitary Groups}, Math. Ann. {\bf 358}, 799--832 (2014).
\bibitem[Ka]{Ka} Kazhdan, D., \textsc{Some Applications of the Weil Representation}, J. Anal. Math. {\bf 32}, 235--248 (1977).
\bibitem[Ku]{Ku} S. Kudla, \textsc{Special Cycles and Derivatives of Eisenstein Series}, in \emph{Heegner Points and Rankin $L$-series}, Math. Sci. Res. Inst. Publ. {\bf 49}, Cambridge University Press, Cambridge (2004).
\bibitem[KM]{KM} Kudla, S., Millson, J., \textsc{Intersection Numbers of Cycles on Locally Symmetric Spaces and Fourier Coefficients of Holomorphic Modular Forms in Several Complex Variables}, Publ. Math. l'IHÉS {\bf 71}, 121--172, (1990).
\bibitem[Li]{Li} Li, C., \textsc{Geometric and Arithmetic Theta Correspondences}, lecture notes, Summer School on the Langlands program, IHES (2022).
\bibitem[Liu]{Liu} Liu, Y., \textsc{Arithmetic Theta Lifting and $L$-Derivatives for Unitary Groups I}, Algebra Number Theory {\bf 5}, 849--921 (2011).
\bibitem[M]{M} Ma, S., \textsc{Zero-Cycles over Zero-Dimensional Cusps}, pre-print, https://arxiv.org/abs/2103.02871 (2021).
\bibitem[MR]{MR} Murty, V. K., Ramakrishnan, D. \textsc{The Albanese of Unitary Shimura Varieties}, in \emph{The Zeta Functions of Picard Modular Surfaces}, eds. R. Langlands and D. Ramakrishnan, 445--464 (1992).
\bibitem[Sch]{Sch} Schneider, P., \textsc{Introduction to the Beilinson Conjectures}, in \emph{Beilinson's Conjecture on Special Values of $L$-Functions}, eds. M. Rapoport, M. Schappacher, and P. Schneider, Oberwolfach Reports (1987).
\bibitem[Sh]{Sh} Shimura, G., \textsc{Automorphic Forms and the Periods of Abelian Varieties}, J. Math. Soc. Japan {\bf 31}, 561--592 (1979).
\bibitem[V]{V} Voisin, C., \textsc{Hodge Theory and Complex Algebraic Geometry II}, Cambridge Studies in Advanced Mathematics {\bf 77}, Cambridge University Press, x+351pp (2003).
\bibitem[Xi]{Xi} Xia, J., \textsc{Some Cases of Kudla’s Modularity Conjecture for Unitary Shimura Varieties}, Forum of Math. Sigma {\bf 10}, paper 37, 1--31 (2022).
\bibitem[Ya]{Ya} Yasaki, D., \textsc{Integral Cohomology of Certain Picard Modular Surfaces}, J. Number Theory {\bf 134}, 13--28 (2014).
\bibitem[YZZ]{YZZ} Yuan, X., Zhang, S., Zhang, W. \textsc{The Gross–-Kohnen–-Zagier Theorem over Totally Real Fields}, Compos. Math. {\bf 145}, 1147--1162 (2009).
\bibitem[Za]{Za} Zagier, D., \emph{Nombres de classes et formes modulaires de poids $3/2$}, C. R. Acad. Sci. Paris S\'er. A-B \textbf{281}, 883--886 (1975).
\bibitem[Ze1]{Ze1} Zemel, S., \textsc{A $p$-adic Approach to the Weil Representation of Discriminant Forms Arising from Even Lattices}, Math. Ann. Qu\'{e}bec {\bf 39}, 61--89 (2015).
\bibitem[Ze2]{Ze2} Zemel, S., \textsc{The Integral Structure of Parabolic Subgroups of Orthogonal Groups}, J. Algebra {\bf 559}, 95--128 (2020).
\bibitem[Ze3]{Ze3} Zemel, S., \textsc{Hermitian Jacobi Forms Having Modules as their Index and Vector-Valued Jacobi Forms}, pre-print, https://arxiv.org/abs/2310.16508 (2023).
\bibitem[Zh]{Zh} Zhang, W., \textsc{Modularity of Generating Functions of Special Cycles on Shimura Varieties}, Ph.D. thesis, Columbia University (2009).

\end{thebibliography}
\end{document}